\documentclass{article}
\usepackage{amsfonts}
\usepackage{amsmath}
\usepackage{amssymb}
\input xy \xyoption{all}

\newtheorem{theorem}{Theorem}[section]

\newtheorem{corollary}[theorem]{Corollary}

\newtheorem{lemma}[theorem]{Lemma}

\newtheorem{proposition}[theorem]{Proposition}

\newtheorem{remark}[theorem]{Remark}
\newenvironment{proof}[1][Proof]{\noindent\textbf{#1.} }{\ \rule{0.5em}{0.5em}}

\begin{document}

\newcommand{\f}{\mathfrak}
\newcommand{\rarrow}{\rightarrow}
\newcommand{\dt}{\left.\frac{d}{dt}\right|_{t=0}}
\newcommand{\wbarnabla}{\widetilde{\bar{\nabla}}}
\newcommand{\wnabla}{\widetilde{\nabla}}
\newcommand{\bb}{\mathbb}

\author{M. Castrill\'on L\'opez \\
ICMAT (CSIC-UAM-UC3M-UCM)\\
Departamento de Geometr\'\i a y Topolog\'\i a \\
Facultad de Matem\'aticas, Universidad Complutense de Madrid\\
28040 Madrid, Spain \and Ignacio Luj\'an  \\Departamento de Geometr\'\i a y Topolog\'\i a \\
Facultad de Matem\'aticas, Universidad Complutense de Madrid\\
28040 Madrid, Spain}

\title{Reduction of Homogeneous Riemannian Structures}
\date{}
\maketitle

\begin{abstract}
The goal of this article is the study of homogeneous Riemannian
structure tensors within the framework of reduction under a group
$H$ of isometries. In a first result, $H$ is a normal subgroup of
the group of symmetries associated to the reducing tensor
$\bar{S}$. The situation when $H$ is any group acting freely is
analyzed in a second result. The invariant classes of homogeneous tensors are also
investigated when reduction is performed. It turns out that the
geometry of the fibres is involved in the preservation of some of
them. Some classical examples illustrate the theory. Finally, the
reduction procedure is applied to fiberings of almost contact
manifolds over almost Hermitian manifolds. If the structure is
moreover Sasakian, the obtained reduced tensor is homogeneous
K\"ahler.
\end{abstract}

\renewcommand{\thefootnote}{\fnsymbol{footnote}}
\renewcommand{\thefootnote}{\arabic{footnote}}
\renewcommand{\thefootnote}{\fnsymbol{footnote}}
\renewcommand{\thefootnote}{\arabic{footnote}}

\section{Introduction}

Since their introduction \cite{AS}, homogeneous structure tensors
has proved to be a powerful tool in the study of homogeneous
Riemannian manifolds. Their nature is twofold. On one hand, they
belong to the tensor algebra. In particular, representation theory
techniques classify them into eight different invariant classes
with respect to a convenient action of the orthogonal group. On
the other hand, homogeneous tensors satisfy a system of partial
differential equations (Ambrose-Singer equations). Many works in
the literature combine these aspects to provide geometric
properties of the underlying Riemannian manifold. The first
characterizations were given to hyperbolic space and naturally
reductive spaces (\cite{TV}). These techniques were subsequently
generalized to Riemannian manifolds with special holonomy by many
authors (for example, \cite{AG}, \cite{CGS0}, \cite{Fin},
\cite{GMM}, \cite{Kir}). It is interesting to point out that there
is not a bijection between tensors and possible groups acting
isometrically and transitively. A same tensor can be defined by
two different groups and a same group can provide different
tensors. In this context, it is remarkable how little is known
about all homogeneous structures and tensors for even well-known
spaces. There is still much work to do.

Manifolds endowed with symmetries are relevant in many situations.
In particular, symmetries represent a classical tool in reduction
schemes intimately related with different topics as systems of
differential equations, variational principles, symplectic or
other geometric structures, etc. In particular, reduction is
recurrently applied in homogeneous manifolds. The goal of this
article is the study of the behaviour of homogeneous tensors by
reduction under subgroups of the group of isometries. In
particular, this gives rise to new homogeneous tensors in the
orbit space of the action. Additionally, the reduction process
reveals and sheds light to some previously known properties of
some homogeneous structures. Finally, the reduction technique
opens a reverse way to get new homogeneous tensors in the
unreduced space from tensors in the orbit space.

The outline of the paper is as follows. In section 2 we recall
basic definitions on homogeneous structure tensors and its
classification. Moreover, the model for reduction will be a
Riemannian principal bundle $\bar{M}\to M$, endowed with the
compatible connection defined as the orthogonal complements to the
fibres. This connection is ubiquitously used for reduction schemes
in Mechanics (see for example, \cite{Mar}, \cite{Mon}) where it is
called the mechanical connection. Section 3 begins with reduction
of homogeneous tensors $\bar{S}$ in $\bar{M}$ by the action under
a normal subgroup $H$ of the group of symmetries $\bar{G}$
associated to $\bar{S}$ (Theorem \ref{Th1}). The space of all
tensors $\bar{S}$ projecting to a same tensor $S$ in $M=\bar{M}/H$
is also determined. The expression of the reduced tensors leads to
a generalization of the reduction result (Theorem \ref{Th2}) to
the case where $\bar{S}$ is not explicitly associated to a precise
group $\bar{G}$. For example, this is the case of non-simply
connected or uncomplete manifolds where the existence of
homogenous tensors still provides interesting geometric
properties. Without the presence of the group $\bar{G}$, the
normality of the structure group $H$ of the bundle $\bar{M}\to M$
needs to be replaced by a suitable differential condition on the
mechanical connection. Finally, the behaviour of the
classification of homogeneous tensors under the reduction process
is analyzed. It is interesting to point out that the geometry of
the orbits of the $H$-action is involved in some of the classes in
this classification. Section 4 provides many examples of the main
results of the article. In particular, they explore the possible
scenarios with respect to the classes when reduction is performed.
Section 5 applies the reduction Theorem to fiberings of almost
contact manifolds over almost Hermitian manifolds (\cite{Ogiue}).
It turns out that the differential condition on the mechanical
connection is automatically satisfied for homogeneous almost
contact or Sasakian tensors. Hence they project to homogeneous
almost Hermitian or K\"ahler tensors in a natural way. This is
connected with other constructions found in the literature (see
\cite{GO}).

\section{Preliminares}

\subsection{Homogeneous Riemannian structures}

Let $(M,g)$ be a connected Riemannian manifold of dimension $n$. Let $\nabla$ be
the Levi-Civita connection of $g$ and $R$ its curvature tensor
with the convention
$$R_{XY}Z=\nabla_X\nabla_YZ-\nabla_Y\nabla_XZ-\nabla_{[X,Y]}Z.$$
A \textit{homogeneous Riemannian structure} on $(M,g)$ is a
$(1,2)$-tensor field $S$ satisfying the so called
\textit{Ambrose-Singer equations}
\begin{equation}\label{ambrose-singer equations}
\widetilde{\nabla}g=0,\quad\widetilde{\nabla}R=0,\quad\widetilde{\nabla}S=0,
\end{equation}
where $\widetilde{\nabla}=\nabla-S$ \cite{TV}. We will also denote
by $S$ the associated $(0,3)$-tensor field obtained by lowering
the contravariant index, $S_{XYZ}=g(S_XY,Z)$.

We now suppose that $(M,g)$ is homogeneous Riemannian. Let $G$ be
a connected Lie group with Lie algebra $\f{g}$ acting effectively
and transitively on $M$ by isometries. And let $K$ be the isotropy
group at a point $x\in M$ with Lie algebra $\f{k}$. A
decomposition $\f{g}=\f{m}\oplus\f{k}$ is said to be a
\textit{reductive decomposition} of $\f{g}$ if
$Ad(K)(\f{m})\subset\f{m}$. Let $\mu$ be the infinitesimal action
of $\f{g}$ at the point $x$, that is
$$\begin{array}{rrcl}
\mu: & \f{g} & \rarrow & T_xM\\
     & \xi   & \mapsto & \dt\Phi_{exp(t\xi)}(x)
\end{array}$$
where $\Phi_{a}$ denotes the action of an element $a\in G$. Then
for all $k\in K$ the following diagram is commutative
\begin{equation}\label{diagrama conmutativo Ad}
 \xymatrix{
\ar@{}[dr]|{\circlearrowleft} \f{g} \ar[d]_-{Ad(k)} \ar[r]^-{\mu}
& T_xM \ar[d]^-{(\Phi_{k})_*}\\ \f{g} \ar[r]^-{\mu} & T_xM }
\end{equation}%
The restriction of $\mu$ to $\f{m}$ gives an isomorphism
$\mu:\f{m}\rarrow T_xM$, and the canonical connection \cite{KN}
$\widetilde{\nabla}$ with respect to the reductive decomposition
$\f{g}=\f{m}\oplus\f{k}$ is determined by its value at $x$
$$\left(\widetilde{\nabla}_XY\right)_x=\mu\left([\mu^{-1}(X),\mu^{-1}(Y)]_{\f{m}}\right),\qquad X,Y\in T_xM.$$
The tensor field $S=\nabla-\widetilde{\nabla}$ is the homogeneous
Riemannian structure associated to the reductive decomposition
$\f{g}=\f{m}\oplus\f{k}$.

Ambrose-Singer Theorem states that a connected, simply
connected and complete Riemannian manifold is homogeneous
Riemannian if and only if it admits a homogeneous structure
tensor. In the case that $(M,g)$ is just a connected Riemannian
manifold, the existence of a homogeneous structure tensor implies
that $(M,g)$ is locally homogeneous. Tricerri and Vanhecke
\cite{TV} gave a classification
of the homogeneous Riemannian
structure tensors in eight invariant classes: the class $\{S=0\}$
of symmetric structures, the total space denoted by $\mathcal{S}$,
three irreducible classes under the action of the group $O(n)$
$$\begin{array}{l}
\mathcal{S}_1= \{S\in\mathcal{S}\left/S_{XYZ}=g(X,Y)\varphi(Z)-g(X,Z)\varphi(Y),\hspace{1mm}\varphi\in\Gamma(T^*M)\right.\}\\
\mathcal{S}_2=\{S\in\mathcal{S}\left/\underset{XYZ}{\f{S}}S_{XYZ}=0,\hspace{1em}c_{12}(S)=0\right.\}\\
\mathcal{S}_3=\{S\in\mathcal{S}\left/S_{XYZ}+S_{YXZ}=0\right.\}
\end{array}$$
and their direct sums
$$\begin{array}{l}
\mathcal{S}_1\oplus\mathcal{S}_2=\{S\in\mathcal{S}\left/\underset{XYZ}{\f{S}}S_{XYZ}=0\right.\}\\
\mathcal{S}_1\oplus\mathcal{S}_3=\{S\in\mathcal{S}/S_{XYZ}+S_{YXZ}=2g(X,Y)\varphi(Z)-g(X,Z)\varphi(Y)\\
\hspace{5.5cm}-g(Y,Z)\varphi(X),\hspace{1mm}\varphi\in\Gamma(T^*M)\}\\
\mathcal{S}_2\oplus\mathcal{S}_3=\{S\in\mathcal{S}\left/c_{12}(S)=0\right.\}
\end{array}$$
where $c_{12}(S)_p(Z)=\sum_iS_{e_ie_iZ}$ for any orthonormal base $\{e_i\}_{i=1,\ldots,n}$ of $T_pM$.

\subsection{The reduced metric in a principal bundle}

Let $\pi :\bar{M}\rightarrow M$ be an $H$-principal bundle, where
$\bar{M}$ is a Riemannian manifold with metric $\bar{g}$ and $H$
acts on $\bar{M}$ by isometries. Although it is not essential, the
action of isometries are
understood as left and hence $\pi $ is a left principal bundle. Let $\bar{x}%
\in \bar{M}$ and let $V_{\bar{x}}\bar{M}$ denote the vertical subspace at $%
\bar{x}$. If we take the orthogonal complement $H_{\bar{x}}\bar{M}=(V_{\bar{x%
}}\bar{M})^{\bot }$ of $V_{\bar{x}}\bar{M}$ in
$T_{\bar{x}}\bar{M}$ with
respect to the metric $\bar{g}$ we have%
\begin{equation}
T_{\bar{x}}\bar{M}=V_{\bar{x}}\bar{M}\oplus H_{\bar{x}}\bar{M}.
\label{descomp con mecanica}
\end{equation}%
Morever, as $H$ acts by isometries, the horizontal subspaces $H_{\bar{x}}%
\bar{M}$ are preserved by the action of $H$, and the decomposition (\ref%
{descomp con mecanica}) leads to the so called \textit{mechanical
connection} in the principal bundle $\bar{M}\rightarrow M$. In
this situation there is a unique Riemannian metric $g$ in $M$ such
that the restriction $\pi _{\ast
}:H_{\bar{x}}\bar{M}\rightarrow T_{\pi (\bar{x})}M$ is an isometry at every $%
\bar{x}\in \bar{M}$. Obviously, the metric $g$ satisfies
\begin{equation}
g(X,Y)\circ \pi =\bar{g}(X^{H},Y^{H})\qquad \forall X,Y\in
\mathfrak{X}(M) \label{metrica reducida}
\end{equation}%
where $X^{H}$ and $Y^{H}$ denote the horizontal lift of $X$ and
$Y$ with respect to the mechanical connection. To complete the
notation, in the
following, for a vector $Z\in T_{\bar{x}}\bar{M}$, we will denote by $%
Z^{h}\in H_{\bar{x}}\bar{M}$ the horizontal part of $Z$ with
respect to the mechanical connection. In particular,
\begin{equation}
Z^{h}=(\pi _{\ast }(Z))^{H}.  \label{pay}
\end{equation}

\bigskip

\begin{proposition}
In the situation above, if $\bar{\nabla}$ is the Levi-Civita
connection for the metric $\bar{g}$, then the Levi-Civita
connection $\nabla $ for the reduced metric $g$ is given by
\begin{equation}
\nabla _{X}Y=\pi _{\ast }(\bar{\nabla}_{X^{H}}Y^{H}),\qquad \forall
X,Y\in \mathfrak{X}(M).  \label{levi-civita reducida}
\end{equation}
\end{proposition}

\begin{proof}
Since the structure group $H$ acts by isometries, it also acts by
affine
transformations of $\bar{\nabla}$. Thus the vector field $\bar{\nabla}%
_{X^{H}}Y^{H}$ is projectable and the operator $D_{X}Y=\pi _{\ast }(\bar{%
\nabla}_{X^{H}}Y^{H})$ is well defined. It is a direct computation
to show
that $D$ fulfills the properties of a linear connection in $M$. For $%
X,Y,Z\in \mathfrak{X}(M)$, from (\ref{metrica reducida}) and
(\ref{pay}) we have
\begin{eqnarray*}
g(D_{X}Y,Z)\circ \pi +g(Y,D_{X},Z)\circ \pi  &=&\bar{g}((\bar{\nabla}%
_{X^{H}}Y^{H})^{h},Z^{H})+\bar{g}(Y^{H},(\bar{\nabla}_{X^{H}}Z^{H})^{h}) \\
&=&\bar{g}(\bar{\nabla}_{X^{H}}Y^{H},Z^{H})+\bar{g}(Y^{H},\bar{\nabla}%
_{X^{H}}Z^{H}) \\
&=&X^{H}(\bar{g}(Y^{H},Z^{H})).
\end{eqnarray*}%
Hence $g(D_{X}Y,Z)+g(Y,D_{X}Z)=X(g(Y,Z))$ and the connection $D$
is metric. Finally, as $[X,Y]^{H}=[X^{H},Y^{H}]^{h}$, the torsion
tensor of $D$ is
\begin{eqnarray*}
T(X,Y) &=&D_{X}Y-D_{Y}X-[X,Y] \\
&=&\pi _{\ast }(\bar{\nabla}_{X^{H}}Y^{H}-\bar{\nabla}%
_{Y^{H}}X^{H}-[X^{H},Y^{H}]) \\
&=&0,
\end{eqnarray*}%
and $D$ is the Levi-Civita connection for $g$.
\end{proof}

\bigskip

\section{Main Results}

\subsection{Reduction by a normal subgroup of isometries}

Let $(\bar{M},\bar{g})$ be a homogeneous Riemannian manifold. Let $\bar{G%
}$ be a group of isometries acting transitively on $\bar{M}$ and
$H\triangleleft \bar{G}$ a normal subgroup acting freely on
$\bar{M}$. The quotient $M=\bar{M}/H$ is thus endowed (cf.
\cite[Th. 9.16]{Lee}) with a smooth structure such that $\pi
:\bar{M}\rightarrow M$ is an $H$-principal bundle. By definition,
the bundle $\pi :\bar{M}\rightarrow M$ is equipped with the
mechanical connection and $M$ is Riemannian with the reduced
metric $g$ as in (\ref{metrica reducida}). Since $H$ is normal,
there is a well-defined action of the group $G=\bar{G}/H$ on $M$
given by

\begin{equation}
\begin{array}{rrcl}
\Phi : & G\times M & \rightarrow  & M \\
& ([\bar{a}],[\bar{x}]) & \mapsto  & \Phi _{\lbrack \bar{a}]}([\bar{x}%
])=[\Phi _{\bar{a}}(\bar{x})]%
\end{array}
\label{reduced action}
\end{equation}%
where $[\bar{a}]$ and $[\bar{x}]$ denotes the classes modulo $H$ of $\bar{a}%
\in \bar{G}$ and $\bar{x}\in \bar{M}$ respectively, and $\Phi
_{\bar{a}}$ denotes the action of $\bar{G}$ on $\bar{M}$. The
action of $G$ is obviously transitive but needs not be effective.
If it is not, we replace $G$
by $G/N$, where $N$ is the kernel of the map $G\rightarrow \mathrm{Isom}(M)$, $%
a\mapsto \Phi _{a}$, $a\in G$.

\begin{proposition}
\label{prop act isometries} The group $G$ acts on $(M,g)$ by
isometries.
\end{proposition}

\begin{proof}
The action (\ref{reduced action}) can be written as $\pi \circ \Phi _{\bar{a}%
}=\Phi _{a}\circ \pi $, for $a=[\bar{a}]$. This implies that
$\bar{G}$ preserves vertical subspaces and, acting by isometries,
also their horizontal complements. Hence, the horizontal lift of
$(\Phi _{a})_{\ast }(X) $ is $(\Phi _{\bar{a}})_{\ast }(X^{H})$
for all $X\in \mathfrak{X}(M)$. In addition, for $X,Y\in
\mathfrak{X}(M)$
\begin{eqnarray*}
g\left((\Phi _{a})_{\ast }(X),(\Phi _{a})_{\ast }(Y)\right)\circ
\pi &=&\bar{g}\left((\Phi
_{a})_{\ast }(X)^{H},(\Phi _{a})_{\ast }(Y)^{H}\right) \\
&=&\bar{g}\left((\Phi _{\bar{a}})_{\ast }(X^{H}),(\Phi
_{\bar{a}})_{\ast }(Y^{H})\right)
\\
&=&\bar{g}\left(X^{H},Y^{H}\right) \\
&=&g\left(X,Y\right)\circ \pi
\end{eqnarray*}%
and then $\Phi _{a}$ is an isometry.
\end{proof}

\bigskip

From this last Proposition, the manifold $(M,g)$ is homogeneous
Riemannian. We will call it the reduced homogeneous
Riemannian manifold.

\begin{remark}\label{remark normality}
\emph{Note that Proposition \ref{prop act isometries} shows that
the horizontal distribution is invariant by $\bar{G}$. This means
that the mechanical connection is $\bar{G}$-invariant, an
important fact that will be used in \S \ref{sec3.2}.}
\end{remark}

Let $\bar{x}\in \bar{M}$ and $x=\pi (\bar{x})\in M$. We denote by
$\bar{K}$ the isotropy group of $\bar{x}$ under the action of
$\bar{G}$, and by $K$ the corresponding isotropy group of $x$
under the action of $G$. We also denote their Lie algebras by
$\bar{\mathfrak{k}}$ and $\mathfrak{k}$ respectively. Then we have

\begin{lemma}
\label{lema isotropia} Let $\tau :\bar{G}\rightarrow G$ be the
quotient
homomorphism. Then $K=\tau (\bar{K})$ and the restriction $\tau |_{\bar{K}}:%
\bar{K}\rightarrow K$ is an isomorphism of groups.
\end{lemma}

\begin{proof}
It is obvious from (\ref{reduced action}) that $\tau
(\bar{K})\subset K$. Let now $k\in K$ and take $\bar{a}\in \bar{G}$
such that $k=\tau (\bar{a})$. Then for any $x\in M$, we have
$x=\Phi _{k}(x)=\pi (\Phi _{\bar{a}}(\bar{x})) $, and then $\Phi
_{\bar{a}}(\bar{x})$ is in the same fibre as $\bar{x}$.
Hence there exists $h\in H$ such that $\Phi _{h}\circ \Phi _{\bar{a}}(\bar{x}%
)=\bar{x}$, so $h\bar{a}\in \bar{K}$. Since $\tau (h\bar{a})=\tau
(\bar{a})=k $ we have $k\in \tau (\bar{K})$. For the injectivity
of $\tau |_{\bar{K}}$,
let $\bar{k}_{1},\bar{k}_{2}\in \bar{K}$ such that $\tau (\bar{k}_{1})=\tau (%
\bar{k}_{2})$. There exists $h\in H$ such that
$h\bar{k}_{1}=\bar{k}_{2}$.
Then $\bar{k}_{1}^{-1}h\bar{k}_{1}=\bar{k}_{1}^{-1}\bar{k}_{2}$, so $\bar{k}%
_{1}^{-1}\bar{k}_{2}\in \bar{K}\cap H$. But since $H$ acts freely, $\bar{k}%
_{1}^{-1}\bar{k}_{2}=\bar{e}$, and then $\bar{k}_{1}=\bar{k}_{2}$.
\end{proof}

\bigskip

%

\begin{theorem}
\label{Th1} Let $(\bar{M},\bar{g})$ be a connected homogeneous
Riemannian
manifold and let $\bar{G}$ be a group of isometries acting transitively and effectively in $%
\bar{M}$. Let $H\lhd \bar{G}$ be a normal subgroup acting freely in $\bar{%
M}$. Then every homogeneous structure tensor $\bar{S}$ associated
to $\bar{G} $ induces a homogeneous structure tensor $S$
associated to $G=\bar{G}/H$ in the reduced Riemannian manifold
$M=\bar{M}/H$.
\end{theorem}

\begin{proof}
Let $\bar{x}\in \bar{M}$ and $x=\pi (\bar{x})\in M$, and let $\bar{\mathfrak{%
g}}$ be the Lie algebra of $\bar{G}$. For any reductive decomposition $\bar{%
\mathfrak{g}}=\bar{\mathfrak{m}}\oplus \bar{\mathfrak{k}}$ associated to $%
\bar{S}$, the restriction isomorphism $\bar{\mu}:\bar{\mathfrak{m}}%
\rightarrow T_{\bar{x}}\bar{M}$ induces from $\bar{g}$ a positive
definite
bilinear form $B$ in $\bar{\mathfrak{m}}$. Moreover, by the commutativity of (\ref{diagrama conmutativo Ad}) the bilinear form $B$ is $Ad(\bar{K})$%
-invariant, that is,
\begin{equation*}
B(Ad(\bar{k})\xi ,Ad(\bar{k})\eta )=B(\xi ,\eta )\qquad \forall
\bar{k}\in \bar{K}.
\end{equation*}%
Then (\ref{descomp con mecanica}) induces an orthogonal and $Ad(\bar{K})$
-invariant decomposition
\begin{equation*}
\bar{\mathfrak{m}}=\bar{\mathfrak{m}}^{v}\oplus
\bar{\mathfrak{m}}^{h},
\end{equation*}%
\label{decomp of m barra} i.e., $Ad(\bar{K})(\bar{\mathfrak{m}}%
^{v})\subset \bar{\mathfrak{m}}^{v}$ and $Ad(\bar{K})(\bar{\mathfrak{m}}%
^{h})\subset \bar{\mathfrak{m}}^{h}$.

Let $\mathfrak{g}=\bar{\mathfrak{g}}/\mathfrak{h}$ be the lie
algebra of $G$ and $\mu :\mathfrak{g}\rightarrow T_{x}M$ the
corresponding infinitesimal action at $x$. For any $\bar{\xi}\in
\bar{\mathfrak{g}}$, by (\ref{reduced action}) we have
\begin{eqnarray*}
\pi _{\ast }\circ \bar{\mu}(\bar{\xi}) &=&\pi _{\ast }\left( \left. \frac{d}{%
dt}\right\vert _{t=0}\Phi _{\exp(t\bar{\xi})}(\bar{x})\right)  \\
&=&\left. \frac{d}{dt}\right\vert _{t=0}\left( \pi \circ \Phi _{\exp(t%
\bar{\xi})}\right) (\bar{x}) \\
&=&\left. \frac{d}{dt}\right\vert _{t=0}\Phi _{\tau (\exp(t\bar{\xi}%
))}(\pi (\bar{x})) \\
&=&\left. \frac{d}{dt}\right\vert _{t=0}\Phi _{\exp(t\tau _{\ast }(\bar{%
\xi}))}(x) \\
&=&\mu \circ \tau _{\ast }(\bar{\xi})
\end{eqnarray*}%
which means that the following diagram is commutative
\begin{equation}\label{diagram
conmutativo tau, mu y pi} \xymatrix{ \ar@{}[dr]|{\circlearrowleft}
\bar{\f{g}} \ar[d]_-{\tau_*} \ar[r]^-{\bar{\mu}} & T_x\bar{M}
\ar[d]^-{\pi_*}\\ \f{g} \ar[r]^-{\mu} & T_pM }
\end{equation}%
Restrictions to $\bar{\mathfrak{m}}^{h}$ and
$\bar{\mathfrak{m}}^{v}$ give commutative diagrams
\begin{equation}
\begin{array}{lr}
\xymatrix{ \ar@{}[dr]|{\circlearrowleft} \bar{\f{m}}^v
\ar[d]_-{\tau_*} \ar[r]^-{\bar{\mu}} & V_x\bar{M}
\ar[d]^-{\pi_*}\\ \tau_*(\bar{\f{m}}^v) \ar[r]^-{\mu} & \{0\} }\ &
\xymatrix{ \ar@{}[dr]|{\circlearrowleft} \bar{\f{m}}^h
\ar[d]_-{\tau_*} \ar[r]^-{\bar{\mu}} & H_x\bar{M}
\ar[d]^-{\pi_*}\\ \tau_*(\bar{\f{m}}^h) \ar[r]^-{\mu} & T_pM }%
\end{array}
\label{diagr conmutativ restring}
\end{equation}%
which shows that $\tau _{\ast }:\mathfrak{m}^{h}\rightarrow \tau _{\ast }(%
\mathfrak{m}^{h})$ and $\mu :\tau _{\ast
}(\mathfrak{m}^{h})\rightarrow T_{x}M$ are isomorphisms, and $\tau
_{\ast }(\mathfrak{m}^{v})\subset \mathfrak{k}$. In addition, by
Lemma \ref{lema isotropia} the restriction of
$\tau _{\ast }:\bar{\mathfrak{g}}\rightarrow \mathfrak{g}$ to $\bar{%
\mathfrak{k}}$ is an isomorphism of Lie algebras from
$\bar{\mathfrak{k}}$ to $\mathfrak{k}$. Therefore, denoting by
$\mathfrak{m}$ the image $\tau_*(\bar{\mathfrak{m}}^h)$, we have
the decomposition
\begin{equation}  \label{decomp reduct abajo}
\mathfrak{g}=\mathfrak{m}\oplus\mathfrak{k}.
\end{equation}

Let $k\in K$ and $\xi\in\mathfrak{m}$, and let $\bar{k}\in\bar{K}$ and $\bar{%
\xi}\in\bar{\mathfrak{m}}^h$ be such that $\tau(\bar{k})=k$ and $\tau_*(\bar{%
\xi})=\xi$ we have
\begin{eqnarray*}
Ad(k)(\xi)& = & Ad(\tau(\bar{k}))(\tau_*(\bar{\xi})) \\
&=& \mu^{-1}\circ\Phi_{\tau(\bar{k})}\circ\mu(\tau_*(\bar{\xi})) \\
&=& \mu^{-1}\circ\Phi_{\tau(\bar{k})}\circ\pi_*(\bar{\mu}(\bar{\xi})) \\
&=& \mu^{-1}\circ\pi_*\circ\Phi_{\bar{k}}(\bar{\mu}(\bar{\xi})) \\
&=& \mu^{-1}\circ\pi_*\circ\bar{\mu}(Ad(\bar{k})(\bar{\xi})) \\
&=& \mu^{-1}\circ\mu\circ\tau_*(Ad(\bar{k})(\bar{\xi})) \\
&=& \tau_*\left(Ad(\bar{k})(\bar{\xi})\right).
\end{eqnarray*}
Since $\bar{\mathfrak{m}}^h$ is $Ad(\bar{K})$-invariant we deduce that $%
Ad(k)(\mathfrak{m})\subset\tau_*(\bar{\mathfrak{m}}^h)=\mathfrak{m}$,
which proves that (\ref{decomp reduct abajo}) is a reductive
decomposition.

The homogeneous structure tensor associated to (\ref{decomp reduct
abajo}) at $x$ is given by \cite[p.24]{TV}
\begin{equation*}
(S_{x})_{X}Y=(\nabla _{Y}\xi ^{\ast })_{x}\qquad X,Y\in T_{x}M
\end{equation*}%
where $\xi ^{\ast }$ is the vector field given by the
infinitesimal action
of $\xi \in \mathfrak{m}$ with $\xi _{x}^{\ast }=\mu (\xi )=X$. Let $\bar{\xi%
}\in \bar{\mathfrak{m}}^{h}$be such that $\tau _{\ast
}(\bar{\xi})=\xi $ then
\begin{eqnarray*}
(S_{x})_{X}Y=(\nabla _{Y}\xi ^{\ast })_{x} &=&\pi _{\ast }\left( (\bar{\nabla%
}_{Y^{H}}(\xi ^{\ast })^{H})_{\bar{x}}\right)  \\
&=&\pi _{\ast }\left( (\bar{\nabla}_{Y^{H}}\bar{\xi}^{\ast
})\right) -\pi
_{\ast }\left( (\bar{\nabla}_{Y^{H}}(\bar{\xi}^{\ast })^{v})_{\bar{x}%
}\right).
\end{eqnarray*}%
Now let $\bar{Z}\in T_{\bar{x}}\bar{M}$ be an horizontal vector,
since $\bar{\xi}_{\bar{x}}^{\ast }$ is horizontal
\begin{equation*}
\bar{g}\left( (\bar{\nabla}_{Y^{H}}(\bar{\xi}^{\ast
})^{v})_{\bar{x}},\bar{Z} \right) =Y^{H}\bar{g} \left(
(\bar{\xi}^{\ast })^{v},\bar{Z}\right) -\bar{g} \left(
(\bar{\xi}^{\ast
})_{\bar{x}}^{v},\bar{\nabla}_{Y^{H}}\bar{Z}\right) =0.
\end{equation*}%
Hence by \cite[p.24]{TV} and (\ref{diagr conmutativ restring})
\begin{equation}
(S_{x})_{X}Y=\pi _{\ast }\left(
(\bar{S}_{\bar{x}})_{X^{H}}Y^{H}\right) \qquad X,Y\in T_{x}M.
\label{tensor reducido en p}
\end{equation}%
Finally we extend $S_{x}$ to the whole $M$ with the action of $G$
to obtain a homogeneous structure tensor $S$.
\end{proof}

\bigskip

We shall call the tensor field $S$ the \textit{reduced homogeneous
structure tensor}.

\begin{corollary}
The reduced homogeneous structure can be expressed as
\begin{equation}  \label{tensor reducido}
S_XY=\pi_*\left(\bar{S}_{X^H}Y^H\right)\qquad
X,Y\in\mathfrak{X}(M).
\end{equation}
\end{corollary}

\begin{proof}
Let $\bar{a}\in\bar{G}$ and $a=\tau(\bar{a})\in G$ we proved that
the
horizontal lift of $(\Phi_a)_*(X)$ is $(\Phi_{\bar{a}})_*(X^H)$ for all $X\in%
\mathfrak{X}(M)$. This together with the invariance of $\bar{S}$
by $\bar{G}$ and the invariance of $S$ by $G$ gives (\ref{tensor
reducido}).
\end{proof}

\bigskip

\subsection{The space of tensors reducing to a given tensor}

Suppose we are now in the situation of Theorem \ref{Th1} and we
have a homogeneous structure tensor $S$ associated to $G$ in the
reduced manifold $M $. Using diagram (\ref{diagram conmutativo
tau, mu y pi}) we can define the subspaces of $\mathfrak{\bar{g}}$
\[
\bar{\mathfrak{m}}^{h}=\tau _{\ast }^{-1}(\mathfrak{m})\cap \bar{\mu}%
^{-1}(H_{\bar{x}}\bar{M})\qquad \text{and}\qquad
\bar{\mathfrak{m}}^{v}=\mathfrak{h}.
\]
Then the decomposition
\begin{equation}
\bar{\mathfrak{g}}=\bar{\mathfrak{m}} \oplus
\bar{\mathfrak{k}},\qquad \text{with} \qquad
\bar{\mathfrak{m}}=\bar{\mathfrak{m}}^{v}\oplus
\bar{\mathfrak{m}}^{h} \label{decomp reductiva construida arriba}
\end{equation}
is a reductive decomposition. Indeed, since $H$ is normal in
$\bar{G}$ it is obvious that $Ad(\bar{K})(\mathfrak{h})\subset
\mathfrak{h}$. On the other hand, for $\bar{k}\in \bar{K}$ and
$\bar{\xi}\in \bar{\mathfrak{m}}^{h}$, as
$\bar{\mu}(Ad(\bar{k})(\bar{\xi}))=(\Phi _{\bar{k}})_{\ast
}(\bar{\mu}(\bar{\xi}))$, we have
$\bar{\mu}(Ad(\bar{k})(\bar{\xi}))\in H_{ \bar{x}}\bar{M}$ and
$\tau _{\ast }\left( Ad(\bar{k})(\bar{\xi})\right) \in
\mathfrak{m}$, and then $Ad(\bar{k})(\bar{\xi})\in
\bar{\mathfrak{m}}^{h}$. The homogeneous structure tensor
associated to this decomposition at $\bar{x}$ is (see, for example
\cite{MGO2})
\begin{equation}\label{formula tensor (3,0)}
(\bar{S}_{\bar{x}})_{\bar{X}\bar{Y}\bar{Z}}=\frac{1}{2} \left(
B([\bar{\xi},\bar{\eta}]_\f{\bar{m}},\bar{\zeta})-B([\bar{\eta},\bar{\zeta}]_\f{\bar{m}},\bar{\xi})
+B([\bar{\zeta},\bar{\xi}]_\f{\bar{m}},\bar{\eta})\right),\qquad
\bar{X},\bar{Y},\bar{Z}\in T_{\bar{x}}\bar{M}
\end{equation}
where $\bar{\xi},\bar{\eta},\bar{\zeta}\in\bar{\f{m}}$ are such
that their images by $\bar{\mu}$ are $X,Y,Z$, and $B$ is the
bilinear form induced on $\bar{\f{m}}$ from $T_{\bar{x}}\bar{M}$
by $\bar{\mu}$. Note that we have exactly the situation in the
proof of Theorem \ref{Th1}, so the homogeneous structure tensor
$\bar{S}$ associated to (\ref{decomp reductiva construida arriba})
reduces to $S$.

We can construct all other homogeneous structures in $\bar{M}$
associated to $\bar{G}$ by changing $\mathfrak{\bar{m}}$ in
(\ref{decomp reductiva construida arriba}) by the graph
\[
\bar{\mathfrak{m}}^\varphi =\{X+\varphi(X)/X\in\mathfrak{\bar{m}}\}
\]
of an $Ad(\bar{K})$-equivariant map $\varphi :\mathfrak{h}\oplus \bar{\mathfrak{m}}^{h}\rightarrow \bar{\mathfrak{%
k}}$. The condition that the new homogeneous structure tensors
reduce to $S$ is equivalent to the condition $\varphi _{\left.
{}\right\vert \bar{\mathfrak{m} }^{h}}=0$. So the family of
homogeneous structure tensors that reduce to $S$ is parameterized
by the set of $Ad(\bar{K})$-equivariant maps $\varphi :
\mathfrak{h}\rightarrow \bar{\mathfrak{k}}$. For the sake of
convenience we will denote by the same $\varphi$ both
$\varphi:\f{h}\rarrow\bar{\f{k}}$ and its extension by zero to
$\bar{\f{m}}=\f{h}\oplus\bar{\f{m}}^h$. The expression of the
homogeneous structure tensor $\bar{S}^{\varphi}$ associated to
this map is the same as in (\ref{formula tensor (3,0)}) by
changing $\f{\bar{m}}$ to $\f{\bar{m}}^\varphi$, $B$ to the
induced bilinear form $B^\varphi$ in $\f{\bar{m}}^\varphi$ and the
$\bar{\xi},\bar{\eta},\bar{\zeta}$ to
$\bar{\xi}'=\bar{\xi}+\varphi({\bar{\xi}})$,
$\bar{\eta}'=\bar{\eta}+\varphi({\bar{\eta}})$,
$\bar{\zeta}'=\bar{\zeta}+\varphi({\bar{\zeta}})\in
\f{\bar{m}}^\varphi$. As
$$[\bar{\xi}',\bar{\eta}']_{\bar{\f{m}}^\varphi}=[\bar{\xi},\bar{\eta}]_{\bar{\f{m}}^\varphi}
+[\bar{\xi},\varphi(\bar{\eta})]_{\bar{\f{m}}^\varphi}+[\varphi(\bar{\xi}),\bar{\eta}]_{\bar{\f{m}}^\varphi}
+[\varphi(\bar{\xi}),\varphi(\bar{\eta})]_{\bar{\f{m}}^\varphi}$$
and
$[\varphi(\bar{\xi}),\varphi(\bar{\eta})]_{\bar{\f{m}}^\varphi}=0$
we have that
\begin{eqnarray*}
B^\varphi\left([\bar{\xi}',\bar{\eta}']_{\bar{\f{m}}^\varphi},\bar{\zeta}'\right)&=&
B^\varphi\left([\bar{\xi},\bar{\eta}]_{\bar{\f{m}}^\varphi},\bar{\zeta}'\right)
+B^\varphi\left([\bar{\xi},\varphi(\bar{\eta})]_{\bar{\f{m}}^\varphi}+[\varphi(\bar{\xi}),\bar{\eta}]_{\bar{\f{m}}^\varphi},\bar{\zeta}'\right)\\
&=&B\left([\bar{\xi},\bar{\eta}]_{\bar{\f{m}}},\bar{\zeta}\right)
+B\left([\bar{\xi},\varphi(\bar{\eta})]+[\varphi(\bar{\xi}),\bar{\eta}],\bar{\zeta}\right),
\end{eqnarray*}
where one has to take into account that the isomorphism
$\f{\bar{m}} \to \f{\bar{m}}^\varphi$, $\bar{\xi}\mapsto
\bar{\xi}+\varphi (\bar{\xi})$ is an isometry with respect to $B$
and $B^\varphi$. Hence
\begin{eqnarray}\label{ecuacio S'}
(\bar{S}^{\varphi}_{\bar{x}})_{\bar{X}\bar{Y}\bar{Z}}&=&(\bar{S}_{\bar{x}})_{\bar{X}\bar{Y}\bar{Z}}
+\frac{1}{2}\left\{ B\left( [\bar{\xi},\varphi (
\bar{\eta})]+[\varphi
(\bar{\xi}),\bar{\eta}],\bar{\zeta}\right)\right.\nonumber \\
&-& \left. B\left( [\bar{\eta},\varphi ( \bar{\zeta})]+[\varphi
(\bar{\eta}),\bar{\zeta}],\bar{\xi}\right)+B\left(
[\bar{\zeta},\varphi ( \bar{\xi})]+[\varphi
(\bar{\zeta}),\bar{\xi}],\bar{\eta}\right)\right\}.
\end{eqnarray}
The summands involving $B$ define a tensor field $P^{\varphi}$
globally defined in $\bar{M}$ by the left action of $\bar{G}$.
More precisely, for any $\bar{y}\in \bar{M}$, with
$\bar{y}=\Phi_{\bar{a}}(\bar{x})$, $\bar{a}\in\bar{G}$, this
tensor is
\begin{eqnarray}\label{Pfi}
(P_{\bar{y}}^{\varphi})_{\bar{X}\bar{Y}\bar{Z}}&=&\frac{1}{2}\left\{
B_{\bar{y}}\left( [\bar{\xi},\varphi_{\bar{y}} (
\bar{\eta})]+[\varphi_{\bar{y}}
(\bar{\xi}),\bar{\eta}],\bar{\zeta}\right) - B_{\bar{y}}\left(
[\bar{\eta},\varphi_{\bar{y}} ( \bar{\zeta})]+[\varphi_{\bar{y}}
(\bar{\eta}),\bar{\zeta}],\bar{\xi}\right) \right.\nonumber\\
& + & \left. B_{\bar{y}}\left( [\bar{\zeta},\varphi_{\bar{y}} (
\bar{\xi})]+[\varphi_{\bar{y}}(\bar{\zeta}),\bar{\xi}],\bar{\eta}\right)\right\},
\end{eqnarray}
for $\bar{X},\bar{Y},\bar{Z}\in T_{\bar{y}}\bar{M}$, where
$$
\bar{\f{m}}_{\bar{y}}:=Ad(\bar{a})(\bar{\f{m}}),\qquad
\bar{\f{k}}_{\bar{y}}:=Ad(\bar{a})(\bar{\f{k}}),
$$
$$\varphi_{\bar{y}}:=Ad(\bar{a})\circ\varphi\circ
Ad(\bar{a}^{-1}):\f{h}\rarrow\bar{\f{k}}_{\bar{y}},$$
$B_{\bar{y}}$ is the bilinear form on $\bar{\f{m}}_{\bar{y}}$
induced from $\bar{g}_{\bar{y}}$ by
$$\bar{\mu}_{\bar{y}}:=\left(\Phi_{\bar{a}}\right)_*\circ\bar{\mu}\circ
Ad(\bar{a}^{-1}):\bar{\f{m}}_{\bar{y}}\rarrow
T_{\bar{y}}\bar{M},$$ and
$\bar{\xi},\bar{\eta},\bar{\zeta}\in\f{\bar{m}}_{\bar{y}}$ are
such that their images by $\bar{\mu}_{\bar{y}}$ are
$\bar{X},\bar{Y},\bar{Z}$ respectively.

\bigskip

\noindent We have then proved

\begin{proposition}
In the situation of Theorem \ref{Th1}, let $S$ be a homogeneous
structure tensor in $M$ associated to $G$. Then the space of
homogeneous structure tensors in $\bar{M}$ associated to $\bar{G}$
and reducing to $S$ is a vector space isomorphic to the space of
$Ad(\bar{K})$-equivariant maps $\varphi :\mathfrak{h}\rightarrow
\bar{\mathfrak{k}}$. Moreover, the isomorphism is given by
\begin{equation*}
\varphi \mapsto \bar{S}^{\varphi}=\bar{S}+P^{\varphi}
\end{equation*}
where $\bar{S}$ is the homogeneous structure associated to the
decomposition (\ref{decomp reductiva construida arriba}) and
$P^\varphi$ is given in (\ref{Pfi}).
\end{proposition}

\bigskip

\subsection{Reduction in a principal bundle}\label{sec3.2}

We have noted in Remark \ref{remark normality} that the normality
of the group $H$ gives the invariance of the mechanical
connection. This implies that the connection form $\omega$ is
$Ad(\bar{G})$-equivariant, i.e.,
\begin{equation}\label{equivariance omega}
\Phi_{\bar{a}}^*\omega=Ad(\bar{a})\cdot\omega ,\qquad
\forall\bar{a}\in\bar{G},
\end{equation}
where
$Ad(\bar{a})\cdot\omega$ denotes the $1$-form in $\bar{M}$ with
values in $\f{h}$ given by
$$(Ad(\bar{a})\cdot\omega)(\bar{X})=Ad(\bar{a})(\omega(\bar{X})).$$
The canonical linear connection $\wbarnabla=\bar{\nabla}-\bar{S}$
of the reductive decomposition
$\bar{\f{g}}=\bar{\f{m}}\oplus\bar{\f{k}}$ at $\bar{x}$ is
characterized by the following property: for every
$\bar{\xi}\in\bar{\f{m}}$, the parallel displacement with respect
to $\wbarnabla$ along the curve
$\gamma(t)=\Phi_{\exp(t\bar{\xi})}(\bar{x})$, from $\bar{x}$ to
$\gamma(t)$, is equal to $(\Phi_{\exp(t\bar{\xi})})_*$ (see
\cite[Vol. II, Ch. X, Corollary 2.5 ]{KN}). Hence, infinitesimally
we have that
$$
\bigl(\wbarnabla_{\bar{X}}\omega\bigr)_{\bar{x}}=ad(\bar{\mu}^{-1}(\bar{X}))\cdot\omega_{\bar{x}},
\qquad \forall \bar{X}\in T_{\bar{x}}\bar{M},
$$
and by the invariance of $\wbarnabla$ by $\bar{G}$
\begin{equation}\label{nabla de omega}
\bigl(\wbarnabla_{\bar{X}}\omega\bigr)_{\bar{y}}=ad(\bar{\mu}_{\bar{y}}^{-1}(\bar{X}))\cdot\omega_{\bar{y}},
\qquad \forall \bar{y}\in\bar{M},\forall \bar{X}\in
T_{\bar{y}}\bar{M},
\end{equation}
that is, the covariant derivative of $\omega$ by the connection
$\wbarnabla$ is proportional to itself by a suitable linear
operator. We note that, in particular, if $H$ is contained in the
center of $\bar{G}$, the linear operator is null, hence $\omega$
is invariant by $\bar{G}$. If $H$ is just a normal subgroup not
contained in the center, condition \eqref{nabla de omega} comes
from the equivariance of $\omega$.


The preceding discussion suggests to study the reduction of
homogeneous structure tensors $\bar{S}$ in a principal bundle
without the use of the group $\bar{G}$. More precisely, in Theorem
\ref{Th1} the group $\bar{G}$ (and its reductive decomposition)
associated to the tensor $\bar{S}$ was a key ingredient. We now
begin with any tensor $\bar{S}$ in a manifold $(\bar{M},\bar{g})$
where a group $H$ acts by isometries (and such that $\bar{M}\to
\bar{M}/H=M$ is a principal bundle) satisfying Ambrose-Singer
equations and an additional algebraic condition for the
mechanical connection analogous to (\ref{nabla de omega}). Then
the tensor $\bar{S}$ can also be projected without using any
reductive decomposition as we can see in the following result.

\begin{theorem}\label{Th2}
Let $(\bar{M},\bar{g})$ be a Riemannian manifold. Let
$\pi:\bar{M}\rarrow M$ be a principal bundle with structure group
$H$ acting on $\bar{M}$ by isometries, and endowed with the
mechanical connection $\omega$. For every $H$-invariant
homogeneous Riemannian structure tensor $\bar{S}$ with canonical
linear connection $\wbarnabla$, if
\begin{equation}\label{condition nabla omega}
\wbarnabla\omega=\alpha\cdot\omega
\end{equation}
for certain 1-form $\alpha$ in $\bar{M}$ taking values in
$\mathrm{End}(\f{h})$, then the tensor field $S$ defined by
\begin{equation}\label{tensor reducido th2}
S_XY=\pi_*\left(\bar{S}_{X^H}Y^H\right)\qquad X,Y\in\f{X}(M)
\end{equation}
is a homogeneous Riemannian structure tensor in $(M,g)$, where $g$
is the reduced Riemannian metric.
\end{theorem}

\begin{proof}
First note that $H$-invariance of $\bar{S}$ implies that
$\bar{S}_{X^H}Y^H$ is projectable and then $S$ is well defined.
Since the structure group $H$ acts by isometries, the Levi-Civita
connection $\bar{\nabla}$ of $\bar{g}$ is $H$-invariant, which
implies that $\wbarnabla=\bar{\nabla}-\bar{S}$ is also
$H$-invariant. Now from condition (\ref{condition nabla omega}) we
have that for all $X,Y\in\f{X}(M)$
$$\omega(\wbarnabla_{X^H}Y^H)=X^H\left(\omega(Y^H)\right)-
\left(\wbarnabla_{X^H}\omega\right)(Y^H)=-\alpha(X^H)\cdot\omega(Y^H)=0,$$
so that $\wbarnabla_{X^H}Y^H$ is horizontal. If we define
$\widetilde{\nabla}=\nabla-S$, $\nabla$ being the Levi-Civita
connection of $g$, then $\wbarnabla_{X^H}Y^H$ projects to
$\widetilde{\nabla}_{X^H}Y^H$. Hence by $H$-invariance,
\begin{equation}\label{wbarnabla de X^H Y^H}
\left(\widetilde{\nabla}_XY\right)^H=\wbarnabla_{X^H}Y^H.
\end{equation}
We now prove that $S$ satisfies Ambrose-Singer equations
(equivalent to those in (\ref{ambrose-singer equations})):
\begin{equation}\label{ambrose-singer equations 2}
\widetilde{\nabla}g=0,\quad\widetilde{\nabla}\widetilde{R}=0,\quad\widetilde{\nabla}S=0,
\end{equation} where $\widetilde{R}$ is the curvature tensor of
$\widetilde{\nabla}$ and $\widetilde{R}$ and $S$ are seen as
$(0,4)$ and $(0,3)$ tensors respectively by lowering their
contravariant index with respect to $g$.

For the first equation, taking into account (\ref{wbarnabla de X^H
Y^H}), we have for $U,X,Y\in\f{X}(M)$
\begin{eqnarray*}
\left(\wnabla_Ug\right)(X,Y)\circ\pi & = &
U\left(g(X,Y)\right)\circ\pi-g(\wnabla_UX,Y)\circ\pi-g(X,\wnabla_UY)\circ\pi\\
& = &
U^H\left(\bar{g}(X^H,Y^H)\right)-\bar{g}\left((\wnabla_UX)^H,Y^H\right)-\bar{g}\left(X^H,(\wnabla_UY)^H\right)\\
& = & U^H\left(\bar{g}(X^H,Y^H)\right)-\bar{g}\left(\wbarnabla_{U^H}X^H,Y^H\right)-\bar{g}\left(X^H,\wbarnabla_{U^H}Y^H\right)\\
& = & \left(\wbarnabla_{U^H}\bar{g}\right)(X^H,Y^H)
\end{eqnarray*}
and then since $\wbarnabla\bar{g}=0$ we have $\wnabla g=0$.

For the third equation, let $U,X,Y,Z\in\f{X}(M)$. Then, again by
(\ref{wbarnabla de X^H Y^H}), we have
\begin{eqnarray*}
\left(\wnabla_US\right)_{XYZ}\circ\pi & = &
U\left(S_{XYZ}\right)\circ\pi-\left(S_{\wnabla_UXYZ}\right)\circ\pi\\
& & -\left(S_{X\wnabla_UYZ}\right)\circ\pi-\left(S_{XY\wnabla_UZ}\right)\circ\pi\\
& = & U^H\left(\bar{S}_{X^HY^HZ^H}\right)-\bar{S}_{(\wnabla_UX)^HY^HZ^H}\\
& & -\bar{S}_{X^H(\wnabla_UY)^HZ^H}-\bar{S}_{X^HY^H(\wnabla_UZ)^H}\\
& = & U^H\left(\bar{S}_{X^HY^HZ^H}\right)-\bar{S}_{\wbarnabla_{U^H}X^HY^HZ^H}\\
& & -\bar{S}_{X^H\wbarnabla_{U^H}Y^HZ^H}-\bar{S}_{X^HY^H\wnabla_{U^H}Z^H}\\
& = & \left(\wbarnabla_{U^H}\bar{S}\right)_{X^HY^HZ^H}
\end{eqnarray*}
which vanishes as $\wbarnabla\bar{S}=0$.

We now prove the second Ambrose-Singer equation. Let
$\widetilde{\bar{R}}$ be the curvature tensor of $\wbarnabla$.
From equation (\ref{wbarnabla de X^H Y^H}), for $X,Y,Z\in\f{X}(M)$
we first have
\begin{eqnarray*}
(\widetilde{R}_{XY}Z)^H & = & \wbarnabla_{X^H}(\wnabla_YZ)^H-\wbarnabla_{Y^H}(\wnabla_XZ)^H-\wbarnabla_{[X,Y]^H}Z^H\\
 & = & \wbarnabla_{X^H}(\wbarnabla_{Y^H}Z^H)-\wbarnabla_{Y^H}(\wbarnabla_{X^H}Z^H)-\wbarnabla_{[X^H,Y^H]^h}Z^H\\
 & = & \wbarnabla_{X^H}(\wbarnabla_{Y^H}Z^H)-\wbarnabla_{Y^H}(\wbarnabla_{X^H}Z^H)-\wbarnabla_{[X^H,Y^H]}Z^H+\wbarnabla_{[X^H,Y^H]^v}Z^H\\
 & = & \widetilde{\bar{R}}_{X^HY^H}Z^H+\wbarnabla_{[X^H,Y^H]^v}Z^H.
\end{eqnarray*}
We shall also denote by $\widetilde{\bar{R}}$ the $(0,4)$ tensor field associated
to $\widetilde{\bar{R}}$
obtained by lowering the contravariant index with respect to $\bar{g}$. Then for
$X,Y,Z,W\in\f{X}(M)$ one has
\begin{eqnarray}
\widetilde{R}_{XYZW}\circ \pi & = &
\widetilde{\bar{R}}_{X^HY^HZ^HW^H}+\bar{g}\left(\wbarnabla_{[X^H,Y^H]^v}Z^H,W^H\right) \label{Rtilde}\\
& = &
\widetilde{\bar{R}}_{X^HY^HZ^HW^H}-\bar{g}\left(\wbarnabla_{\Omega(X^H,Y^H)^*}Z^H,W^H\right)\nonumber,
\end{eqnarray}
where $\Omega(X^H,Y^H)^*$ is the fundamental vector field
associated to $\Omega(X^H,Y^H)\in\f{h}$. For any
$\bar{x}\in\bar{M}$, let $\mathbb{I}(\bar{x})$ the bilinear form
in $\f{h}$ defined as
$$\bb{I}(\bar{x})(\xi,\eta)=\bar{g}(\xi^*_{\bar{x}},\eta^*_{\bar{x}}),\qquad \forall\xi,\eta\in\f{h}.$$
Applying Koszul's formula for $\bar{\nabla}$ and taking into
account that $[X^H,\xi ^*]=0$ for any $X\in \mathfrak{X}(M)$, $\xi
\in \f{h}$, we have
\begin{eqnarray*}
\bar{g}\left(\wbarnabla_{\Omega(X^H,Y^H)^*}Z^H,W^H\right) & = &
\bar{g}\left(\bar{\nabla}_{\Omega(X^H,Y^H)^*}Z^H,W^H\right)-\bar{g}\left(\bar{S}_{\Omega(X^H,Y^H)^*}Z^HW^H\right)\\
& = &
\frac{1}{2}\bb{I}\left(\Omega(X^H,Y^H),\Omega(Z^H,W^H)\right)-\bar{S}_{\Omega(X^H,Y^H)^*Z^HW^H},
\end{eqnarray*}
where, as usual, $\wbarnabla=\bar{\nabla}-\bar{S}$. Then applying
the previous equation and equation (\ref{Rtilde}), a direct
computation shows that
\begin{eqnarray}\label{formulote curvatura}
\left(\wnabla_U\widetilde{R}\right)_{XYZW}\circ\pi& = &
\left(\wbarnabla_{U^H}\widetilde{\bar{R}}\right)_{X^HY^HZ^HW^H}\nonumber\\
& -&
\frac{1}{2}U^H\left(\bb{I}(\Omega(X^H,Y^H),\Omega(Z^H,W^H))\right)\nonumber\\
& +& \frac{1}{2}\bb{I}\left(\Omega(\wbarnabla_{U^H}X^H,Y^H),\Omega(Z^H,W^H)\right)\nonumber\\
& + &
\frac{1}{2}\bb{I}\left(\Omega(X^H,\wbarnabla_{U^H}Y^H),\Omega(Z^H,W^H)\right)\nonumber\\
& + & \frac{1}{2}\bb{I}\left(\Omega(X^H,Y^H),\Omega(\wbarnabla_{U^H}Z^H,W^H)\right)\\
& + &
\frac{1}{2}\bb{I}\left(\Omega(X^H,Y^H),\Omega(Z^H,\wbarnabla_{U^H}W^H)\right)\nonumber\\
& + &
U^H\left(\bar{S}_{\Omega(X^H,Y^H)^*Z^HW^H}\right)-\bar{S}_{\Omega(\wbarnabla_{U^H}X^H,Y^H)^*Z^HW^H}\nonumber\\
& - &
\bar{S}_{\Omega(X^H,\wbarnabla_{U^H}Y^H)^*Z^HW^H}-\bar{S}_{\Omega(X^H,Y^H)^*(\wbarnabla_{U^H}Z^H)W^H}\nonumber\\
& - & \bar{S}_{\Omega(X^H,Y^H)^*Z^H(\wbarnabla_{U^H}W^H)}.\nonumber
\end{eqnarray}
On the other hand, by (\ref{condition nabla omega})
$$0=\left(\wbarnabla_{X^H}\omega\right)(Y^H)-\left(\wbarnabla_{Y^H}\omega\right)(X^H)=
d\omega(X^H,Y^H)-\omega\left(\widetilde{\bar{T}}_{X^H}Y^H\right),$$
where $\widetilde{\bar{T}}$ is the torsion tensor field of
$\wbarnabla$. Then, since by definition $\Omega(\bar{X},\bar{Y})=d\omega(\bar{X}^h,\bar{Y}^h)$, we have
$$\Omega(X^H,Y^H)=\omega\left(\widetilde{\bar{T}}_{X^H}Y^H\right).$$
Using that
$$\widetilde{\bar{T}}_{X^H}Y^H=\bar{S}_{Y^H}X^H-\bar{S}_{X^H}Y^H,$$
and conditions (\ref{condition nabla omega}) and $\wbarnabla\bar{S}=0$, one has that
\begin{equation}\label{condition nabla Omega}
\left(\wbarnabla_{U^H}\Omega\right)(X^H,Y^H)=\alpha(U^H)\cdot\Omega(X^H,Y^H).
\end{equation}
Now, from $\omega([X^H,Y^H]^v)=-\Omega(X^H,Y^H)$ and
(\ref{condition nabla omega}) we get
\begin{equation}
\label{conditionn}
\omega\left(\wbarnabla_{U^H}[X^H,Y^H]^v\right)=-U^H\left(\Omega(X^H,Y^H)\right)+\alpha(U^H)\cdot\Omega(X^H,Y^H),
\end{equation}
and hence we have
\begin{eqnarray*}
U^H\left(\bb{I}\left(\Omega(X^H,Y^H),\Omega(Z^H,W^H)\right)\right)
& = &
\bar{g}\left(\wbarnabla_{U^H}[X^H,Y^H]^v,[Z^H,W^H]^v\right)\\
& + &
\bar{g}\left([X^H,Y^H]^v,\wbarnabla_{U^H}[Z^H,W^H]^v\right)\\
& = & \bb{I}\left(U^H\Omega(X^H,Y^H),\Omega(Z^H,W^H)\right)\\
& - &
\bb{I}\left(\alpha(U^H)\cdot\Omega(X^H,Y^H),\Omega(Z^H,W^H)\right)\\
& + & \bb{I}\left(\Omega(X^H,Y^H),U^H\Omega(Z^H,W^H)\right)\\
& - &
\bb{I}\left(\Omega(X^H,Y^H),\alpha(U^H)\cdot\Omega(Z^H,W^H)\right).
\end{eqnarray*}
In addition, by (\ref{condition nabla Omega}) and
(\ref{conditionn})
$$\Omega(\wbarnabla_{U^H}X^H,Y^H)+\Omega(X^H,\wbarnabla_{U^H}Y^H)=-\omega\left(\wbarnabla_{U^H}[X^H,Y^H]^v\right),$$
so
\begin{equation}
\Omega(\wbarnabla_{U^H}X^H,Y^H)^*+\Omega(X^H,\wbarnabla_{U^H}Y^H)^*=\wbarnabla_{U^H}\Omega(X^H,Y^H)^*
\label{kondition}
\end{equation}
since $\wbarnabla_{U^H}[X^H,Y^H]^v$ is vertical. Substituting the
preceding formulas and grouping terms, (\ref{formulote curvatura})
becomes
\begin{eqnarray*}
\left(\wnabla_U\widetilde{R}\right)_{XYZW}\circ\pi& = &
\left(\wbarnabla_{U^H}\widetilde{\bar{R}}\right)_{X^HY^HZ^HW^H}\\
& + &
\frac{1}{2}\bb{I}\left((\wbarnabla_{U^H}\Omega)(X^H,Y^H),\Omega(Z^H,W^H)\right)\\
& - &
\frac{1}{2}\bb{I}\left(\alpha(U^H)\cdot\Omega(X^H,Y^H),\Omega(Z^H,W^H)\right)\\
& + & \frac{1}{2}\bb{I}\left(\Omega(X^H,Y^H),(\wbarnabla_{U^H}\Omega)(Z^H,W^H)\right)\\
& - & \frac{1}{2}\bb{I}\left(\Omega(X^H,Y^H),\alpha(U^H)\cdot\Omega(Z^H,W^H)\right)\\
& - &
\left(\wbarnabla_{U^H}\bar{S}\right)_{\Omega(X^H,Y^H)^*Z^HW^H},\\
\end{eqnarray*}
from where, taking into account (\ref{condition nabla Omega}) and
\eqref{kondition}, we deduce that $\wnabla_U\widetilde{R}=0$. This
finishes the proof of Theorem \ref{Th2}.
\end{proof}

\bigskip

\begin{remark}
\emph{In the situation of Theorem \ref{Th2}, in the case $\bar{S}$
is a homogeneous structure tensor associated to a Lie group $\bar{G}$ acting by
isometries in $\bar{M}$, one could ask if $H$ can be seen as a
normal subgroup of $\bar{G}$ and if the projected tensor $S$ is
associated to the group $G=\bar{G}/H$. The answer is not
necessarily positive. More precisely, for a connected, simply
connected and complete manifold $\bar{M}$, if we construct the
group $\bar{G}$ from $\bar{S}$ following the proof of
Ambrose-Singer Theorem (as in \cite{TV}), one can see that the
normality of $H$ is not guaranteed and the group $\bar{G}$ needs
not project to the group $G$ constructed in $M$ from $S$ by the
same method. An example of this situation will be shown in \S
\ref{sectHopf} (Hopf fibration case $\lambda =0$).}
\end{remark}

\begin{remark}
\emph{The algebraic condition  \eqref{condition nabla omega} for
$\alpha =0$ is an invariance condition and can be implemented in
Ambrose-Singer conditions as in Kiri\v{c}enko's theorem (see
\cite{Kir}). This situation can be found in the last section of
the present paper in the framework of almost contact metric
homogeneous structures, where this condition is automatically
satisfied. Note that for non trivial $\alpha$, the situation would
require an equivariant version of this theorem.}
\end{remark}

\subsection{Reduction and homogeneous classes}

In the situation of Theorem \ref{Th2}:

\begin{proposition}\label{conservation}
The classes $\{0\}$, $\mathcal{S}_1$, $\mathcal{S}_3$,
$\mathcal{S}_1\oplus\mathcal{S}_2$ and
$\mathcal{S}_1\oplus\mathcal{S}_3$ are invariant under the
reduction procedure.
\end{proposition}

\begin{proof}
By the expression of the reduced structure tensor (\ref{tensor
reducido th2}) it is obvious that if $\bar{S}=0$ then $S=0$. Let
$\bar{S}\in \mathcal{S}_1$ given by the expression
$$\bar{S}_{\bar{X}\bar{Y}\bar{Z}}=\bar{g}(\bar{X},\bar{Y})\bar{g}(\bar{\xi},\bar{Z})
-\bar{g}(\bar{Y},\bar{\xi})\bar{g}(\bar{X},\bar{Z})$$
where $\bar{\xi}$ is a vector field parallel with respect to
$\wbarnabla$. Since $\bar{S}$ is $H$-invariant the vector field
$\bar{\xi}$ is also $H$-invariant, and then projectable. Let $\xi$
be the projection of $\bar{\xi}$ we have $\xi^H=\bar{\xi}^h$ and
then
\begin{eqnarray*}
S_{XYZ}\circ\pi & = &
\bar{g}(X^H,Y^H)\bar{g}(\bar{\xi},Z^H)-\bar{g}(Y^H,\bar{\xi})\bar{g}(X^H,Z^H)\\
& = &
\bar{g}(X^H,Y^H)\bar{g}(\xi^H,Z^H)-\bar{g}(Y^H,\xi^H)\bar{g}(X^H,Z^H)\\
& = & g(X,Y)g(\xi,Z)\circ\pi-g(Y,\xi)g(X,Z)\circ\pi
\end{eqnarray*}
hence $S\in\mathcal{S}_1$. With a similar argument one proves that
the class $\mathcal{S}_1\oplus\mathcal{S}_2$ is also invariant.
For the classes $\mathcal{S}_3$ and
$\mathcal{S}_1\oplus\mathcal{S}_3$, they are characterized by
algebraic conditions clearly preserved by the reduction formula
(\ref{tensor reducido th2}).
\end{proof}

\bigskip

The other two classes $\mathcal{S}_2$ and
$\mathcal{S}_2\oplus\mathcal{S}_3$ are characterized by the
vanishing of the trace $c_{12}$. Let $x\in M$ and
$\{e_i\}_{i=1,...,n}$ be an orthonormal base of $T_x M$, then for
$X\in T_x M$
\begin{equation}\label{formula c12}c_{12}(S)(X)=\sum_i S_{e_ie_iX}=
\sum_i \bar{S}_{e_i^He_i^HX^H}=c_{12}(\bar{S})(X^H)-\sum_j\bar{S}_{V_jV_jX^H},
\end{equation}
where $\{V_j\}_{j=1,...,r}$ is an orthonormal basis of the
vertical subspace $V_{\bar{x}}\bar{M}$ , $\bar{x}\in \pi
^{-1}(x)$. From $\wbarnabla=\wnabla-\bar{S}$ one has
$$\bar{S}_{V_jV_jX^H}=\bar{g}(\wnabla_{V_j}V_j,X^H)-\bar{g}(\wbarnabla_{V_j}V_j,X^H)=
-\bar{g}(\wnabla_{V_j}X^H,V_j)+\bar{g}(\wbarnabla_{V_j}X^H,V_j),$$
where the vectors $V_j$, $j=1,...,r$, are extended to unitary and
respectively orthogonal vertical vector fields. As from
\eqref{condition nabla omega} we have
$$\omega(\wbarnabla_{V_j}X^H)=V_j(\omega(X^H))-\alpha(V_j)\cdot\omega(X^H)=0,$$
the second summand in the formula for $\bar{S}_{V_jV_jX^H}$ is
zero, and then
$$\bar{S}_{V_jV_jX^H}=-\bar{g}(\wnabla_{V_j}X^H,V_j)=\bar{g}(B(V_j,V_j),X^H),$$
where $B$ denotes the second fundamental form of the fibre $\pi
^{-1}(x)$ at $\bar{x}$. Inserting this in (\ref{formula c12}) we
obtain that
$$c_{12}(S)(X)=c_{12}(\bar{S})(X^H)-\sum_j\bar{g}(B(V_j,V_j),X^H)=
c_{12}(\bar{S})(X^H)-\bar{g}(\mathrm{H},X^H)$$ where $\mathrm{H}$
denotes the mean curvature operator (trace of $B$) of the fibre at
$\bar{x}$. We have proved the following.

\begin{proposition}\label{proposicion minimal}
The classes $\mathcal{S}_2$ and $\mathcal{S}_2\oplus\mathcal{S}_3$
are invariant under reduction if and only if the fibres of the
principal bundle $\pi: (\bar{M},\bar{g})\to(M,g)$ are minimal
Riamannian sub-manifolds of $(\bar{M},\bar{g})$.
\end{proposition}

\begin{remark}
\emph{Proposition \ref{conservation} and \ref{proposicion minimal}
(when the fibres are minimal) do not exclude that a homogeneous
structure tensor $\bar{S}$ in a class
$\mathcal{S}_i\oplus\mathcal{S}_j$ reduces to a tensor $S$
belonging to classes $\mathcal{S}_i$ or $\mathcal{S}_j$, or even
to the null tensor. We shall show some examples of this situations
in the next section.}
\end{remark}

\section{Examples}

\subsection{Real hyperbolic space}

The real $n$-dimensional hyperbolic space
$(\mathbb{R}H(n),\bar{g})$
$$\mathbb{R}H(n)=\{(\bar{y}^0,\bar{y}^1,\ldots,\bar{y}^{n-1})\in \mathbb{R}^n/\bar{y}^0>0\}$$
$$\bar{g}=\frac{1}{(\bar{y}^0)^2}\sum_{j=0}^{n-1}d\bar{y}^j\otimes d\bar{y}^j,$$
is a symmetric space, $\mathbb{R}H(n)=SO(n-1,1)/O(n-1)$. If we
consider the Iwasawa decomposition of its full Lie group of
isometries
$$SO(1,n-1)=O(n-1)AN,$$ then we can identify $\mathbb{R}H(n)\simeq AN$ so that the hyperbolic space has a solvable Lie
group structure given by
$$(\bar{x}^0,\bar{y}^1,\ldots,\bar{x}^{n-1})\cdot(\bar{y}^0,\bar{y}^1,\ldots,\bar{y}^{n-1})=
(\bar{x}^0\bar{y}^0,\bar{x}^0\bar{y}^1+\bar{x}^1,\ldots,\bar{x}^0\bar{y}^{n-1}+\bar{x}^{n-1}).$$
Hence the real hyperbolic space acts freely, transitively and by
isometries on itself by left translations. The homogeneous
structure tensor $\bar{S}$ associated to this action (see
\cite{TV}) is a $\mathcal{S}_1$ structure given by
$$\bar{S}_{\bar{X}\bar{Y}\bar{Z}}=\bar{g}(\bar{X},\bar{Y})\bar{g}(\bar{\xi},\bar{Y})-\bar{g}(\bar{\xi},\bar{Y})\bar{g}(\bar{X},\bar{Z}),
\hspace{1em}\bar{X},\bar{Y},\bar{Z}\in\f{X}(\mathbb{R}H(n))$$
where
$$\bar{\xi}=\bar{y}^0\frac{\partial}{\partial\bar{y}^0}.$$
Let $H_i\simeq \mathbb{R}$, $i=2,\ldots,n-1$, be the normal
subgroups of $\mathbb{R}H(n)$ given by
$$H_i=\{(1,0,\ldots,\lambda,0,\ldots,0)/\lambda\in\mathbb{R}\}$$
where $\lambda$ is in the $i$-th position. Reduction by the action
of $H_i$ gives the fibration
$$\begin{array}{rcl}
\mathbb{R}H(n) & \rarrow & \mathbb{R}H(n-1)\\
(\bar{y}^0,\ldots,\bar{y}^{n-1}) & \mapsto & (\bar{y}^0,\ldots,\bar{y}^{i-1},\bar{y}^{i+1}\ldots,\bar{y}^{n-1})
\end{array}$$
with vertical and horizontal subspaces at $\bar{y}\in\mathbb{R}H(n)$
$$\begin{array}{l}V_{\bar{y}}\mathbb{R}H(n)=\mathrm{span}\left\{\frac{\partial}{\partial\bar{y}^i}\right\},\\
H_{\bar{y}}\mathbb{R}H(n)=\mathrm{span}\left\{\frac{\partial}{\partial\bar{y}^0},\ldots,
\frac{\partial}{\partial\bar{y}^{i-1}},\frac{\partial}{\partial\bar{y}^{i+1}},
\ldots,\frac{\partial}{\partial\bar{y}^{n-1}}\right\}.\end{array}$$ Hence the
induced metric on $\mathbb{R}H(n-1)$ is
$$g=\frac{1}{(y^0)^2}\sum_{j=0}^{n-2}dy^j\otimes dy^j$$
where $(y^0,\ldots,y^{n-2})$ are the natural coordinates of
$\mathbb{R}H(n-1)$. As a straightforward computation shows, the
reduced homogeneous structure tensor $S$ is
$$S_{XYZ}=g(X,Y)g(\xi,Z)-g(\xi,Y)g(X,Z),\hspace{1em}X,Y,Z\in\f{X}(\mathbb{R}H(n-1))$$
where
$$\xi=y^0\frac{\partial}{\partial y^0}.$$
We have proved that the reduction $\mathbb{R}H(n)\to
\mathbb{R}H(n-1)$ sends the canonical tensor associated to the
solvable structure of the $n$-dimensional hyperbolic space to the canonical tensor associated to the
solvable structure of the $n-1$-dimensional hyperbolic space. The
reduction procedure has then preserved the $\mathcal{S}_1$ class in
this case.

\bigskip

We now confine ourselves to the $4$-dimensional hyperbolic space.
Besides its symmetric description, all other groups of isometries
acting transitively are of the type (cf. \cite{CGS}) $\bar{G}=FN$,
where $F$ is a connected closed subgroup of $SO(3)A$ with
nontrivial projection to $A$. In particular, we now consider
$$\bar{G}=SO(2)AN.$$
Geometrically, if we see $SO(2)$ as the isotropy group of the point
$\bar{x}=(1,0,0,0)$, its Lie algebra $\bar{\f{k}}$ are
infinitesimal rotations generated by
$$r= \bar{y}^2\frac{\partial}{\partial \bar{y}^3}-\bar{y}^3 \frac{\partial}{\partial \bar{y}^2}.$$
The subspace $\bar{\f{m}}=\f{a}\oplus\f{n}$, which is the lie
algebra of the factor $AN$, gives a reductive decomposition
$$\bar{\f{g}}=\bar{\f{m}}\oplus\bar{\f{k}}.$$
Let $a\in\f{a}$, $n_1,n_2,n_3\in\f{n}$ be the generators of
$\f{a}$ and $\f{n}$ respectively, where $n_i$ is the infinitesimal
translation in $\mathbb{R}H(4)$ in the direction of $\partial
/\partial \bar{y}^i$. All other reductive decompositions
$\bar{\f{g}}=\bar{\f{m}}^{\varphi}+\bar{\f{k}}$ associated to
$\bar{\f{g}}$ and $\bar{\f{k}}$ are given by the graph of any
equivariant map $\varphi : \f{m}\to \f{k}$. As a computation
shows, all these equivariant maps are
$$\begin{array}{rrcl}
\varphi_{(\lambda_0,\lambda_1)}: & \f{m} & \to & \f{k}\\
                                 & a & \mapsto & \lambda_0r\\
                                 & n_1 & \mapsto & \lambda_1r\\
                                 & n_2,n_3 & \mapsto & 0,\\
\end{array}$$
with $\lambda _0 ,\lambda _1 \in \bb{R}$. The homogeneous
structure tensors associated to this 2-parameter family of
reductive decompositions are
$$\bar{S}^{(\lambda_0,\lambda_1)}=\frac{1}{(\bar{y}^0)^3}\left(\sum_{k=1}^3d\bar{y}^k\otimes d\bar{y}^k\wedge d\bar{y}^0
-\lambda_0d\bar{y}^0\otimes d\bar{y}^2\wedge
d\bar{y}^3-\lambda_1d\bar{y}^1\otimes d\bar{y}^2\wedge
d\bar{y}^3\right),$$ and the canonical connection
$\wbarnabla=\bar{\nabla}-\bar{S}^{(\lambda_0,\lambda_1)}$ (where
$\bar{\nabla}$ is the Levi-Civita connection of $\bar{g}$) is then
given by
$$\begin{array}{c}
\wbarnabla_{\partial_0}\partial_0=-\frac{1}{\bar{y}^0}\partial_0,\quad
\wbarnabla_{\partial_0}\partial_1=-\frac{1}{\bar{y}^0}\partial_1,\quad
\wbarnabla_{\partial_0}\partial_2=-\frac{1}{\bar{y}^0}\partial_2+\frac{\lambda_0}{\bar{y}^0}\partial_3,\\
\wbarnabla_{\partial_0}\partial_3=-\frac{1}{\bar{y}^0}\partial_3-\frac{\lambda_0}{\bar{y}^0}\partial_2,\quad
\wbarnabla_{\partial_1}\partial_2=\frac{\lambda_1}{\bar{y}^0}\partial_3,\quad
\wbarnabla_{\partial_1}\partial_3=-\frac{\lambda_1}{\bar{y}^0}\partial_2,
\end{array}$$
where $\partial_k$ stands for
$\frac{\partial}{\partial\bar{y}^k}$. Let $H\simeq \bb{R}$ be the
subgroup of $\mathbb{R}H(4)$ given by
$$H=\{(1,\lambda,0,0)/\lambda\in\mathbb{R}\}.$$
We take the $H$-principal bundle
$$\begin{array}{rcl}
\mathbb{R}H(4) & \rarrow & \mathbb{R}H(3)\\
(\bar{y}^0,\bar{y}^{1},\bar{y}^{2},\bar{y}^{3}) & \mapsto &
(\bar{y}^0,\bar{y}^2,\bar{y}^3)
\end{array}$$
with mechanical connection form $\omega=d\bar{y}^1$. We have that
$$\wbarnabla \omega =\left(\frac{1}{\bar{y}^0}d\bar{y}^0\right)\cdot\omega$$
where we have identified $\f{h}\simeq\bb{R}$ and
$\mathrm{End}(\f{h})\simeq\bb{R}$. From Theorem \ref{Th2}, the
family of homogeneous structure tensors
$\bar{S}^{(\lambda_0,\lambda_1)}$ can then be reduced to
$\bb{R}H(3)$. If $(y^0,y^1,y^2)$ are the standard coordinates of
$\bb{R}H(3)$, these reduced homogeneous structure tensors form a
one-parameter family
$$S^{\lambda_0}=\frac{1}{(y^0)^3}\left(\sum_{k=1}^2 dy^k\otimes dy^k\wedge dy^0-\lambda_0dy^0\otimes dy^1\wedge dy^2\right).$$
Note that in the expression of both
$\bar{S}^{(\lambda_0,\lambda_1)}$ and $S^{\lambda_0}$ the first
summand is the standard $\mathcal{S}_1$ structure of $\bb{R}H(4)$
and $\bb{R}H(3)$ respectively. The other summands are of type
$\mathcal{S}_2\oplus \mathcal{S}_3$ since they have null trace,
which makes $\bar{S}^{(\lambda_0,\lambda_1)}$ and $S^{\lambda_0}$
of type $\mathcal{S}_1\oplus\mathcal{S}_2\oplus \mathcal{S}_3$ in
the generic case. In the especial case $\lambda_0=0$ we will have
a reduction of the generic class
$\mathcal{S}_1\oplus\mathcal{S}_2\oplus \mathcal{S}_3$ to the
class $\mathcal{S}_1$. This example can be generalized to the
principal bundle $\bb{R}H(n)\to\bb{R}H(n-1).$

\subsection{Hopf Fibrations}

\subsubsection{The fibration $S^3 \to S^2$}\label{sectHopf}

Let $S^3\subset\mathbb{R}^4\simeq \mathbb{C}^2$ be the $3$-sphere
with its standard Riemannian metric with full isometry group
$O(4)$. The natural action of $U(2)$ in $\mathbb{C}^2$ defines a
transitive and effective action of $U(2)$ on $S^3$ given by
$$\begin{array}{rcl}
U(2) & \hookrightarrow & SO(4)\\
\begin{pmatrix}a&b\\c&d\end{pmatrix} & \mapsto & \begin{pmatrix}\mathrm{Re}(a)
&-\mathrm{Im}(a)&\mathrm{Re}(b)&-\mathrm{Im}(b)\\\mathrm{Im}(a)
&\mathrm{Re}(a)&\mathrm{Im}(b)&\mathrm{Re}(b)\\\mathrm{Re}(c)
&-\mathrm{Im}(c)&\mathrm{Re}(d)&-\mathrm{Im}(d)\\\mathrm{Im}(c)&\mathrm{Re}(c)&\mathrm{Im}(d)&\mathrm{Re}(d)\end{pmatrix}
\end{array}.$$
The isotropy group at $\bar{x}=(1,0,0,0)\in S^3$ is
$$\bar{K}=\left\{\begin{pmatrix}1 & 0\\ 0 & z\end{pmatrix}\in U(2)/z\in U(1)\right\}$$
with lie algebra
$$\bar{\f{k}}=\mathrm{span}\left\{\begin{pmatrix}0 & 0\\ 0 & i\end{pmatrix}\right\}.$$
It is easy to see that the complement
$$\bar{\f{m}}=\mathrm{span}\left\{\begin{pmatrix}0 & 1\\-1 & 0\end{pmatrix},
\begin{pmatrix}0 & i\\i & 0\end{pmatrix},\begin{pmatrix}i & 0\\0 & -i\end{pmatrix}\right\}$$
makes $\f{u}(2)=\bar{\f{m}}\oplus\bar{\f{k}}$ a reductive
decomposition. The rest of complements $\bar{\f{m}}'$ giving
reductive decompositions $\f{u}(2)=\bar{\f{m}}'\oplus\bar{\f{k}}$
are obtained as the graph of $Ad(\bar{K})$-equivariant maps
$\varphi:\bar{\f{m}}\rarrow\bar{\f{k}}$. One can check that these
decompositions are exhausted by the following one-parameter family
of complements
$$\bar{\f{m}}_{\lambda}=\mathrm{span}\left\{\begin{pmatrix}0 & 1\\-1
& 0\end{pmatrix},\begin{pmatrix}0 & i\\i &
0\end{pmatrix},\begin{pmatrix}i & 0\\0 &
-i\end{pmatrix}+\lambda\begin{pmatrix}0 & 0\\0 &
i\end{pmatrix}\right\},\hspace{1em}\lambda\in\mathbb{R}.$$ From
formula (\ref{formula tensor (3,0)}), the expression of the
homogeneous structure tensor $\bar{S}^{\lambda}$ associated to
each reductive decomposition computed at $T_{\bar{x}}S^3$ is given
by
\begin{equation}\label{tensor S3 u(2)}
(\bar{S}^{\lambda})_{\bar{x}}=(\lambda-1)d\bar{x}^2\otimes
d\bar{x}^3\wedge d\bar{x}^4+d\bar{x}^3\otimes d\bar{x}^2\wedge
d\bar{x}^4-d\bar{x}^4\otimes d\bar{x}^2\wedge
d\bar{x}^3,\end{equation} where
$(\bar{x}^1,\bar{x}^2,\bar{x}^3,\bar{x}^4)$ is the natural system
of coordinates in $\mathbb{R}^4$.

Let $H$ be the subgroup of $U(2)$ isomorphic to $U(1)$ given by
$$H=\left\{\begin{pmatrix}z&0\\0&z\end{pmatrix}/z\in U(1)\right\}.$$
It is easy to check that $H$ is a normal subgroup of $U(2)$
acting freely on $S^3$. Reduction  by the action of $H$ gives the
Hopf fibration $S^3 \rarrow S^2$ with vertical and
horizontal subspaces at $\bar{x}$
$$V_{\bar{x}}S^3=\mathrm{span}\left\{\frac{\partial}{\partial\bar{x}_2}\right\},
\hspace{1em}H_{\bar{x}}S^3=\mathrm{span}\left\{\frac{\partial}{\partial\bar{x}_3},\frac{\partial}{\partial\bar{x}_4}\right\}.$$
Since all the terms of $\bar{S}^{\lambda}$ have the vertical
factor $d\bar{x}^2$, it is obvious that they all reduce to the
structure tensor $S=0$ on $S^2$, describing $S^2$ as a symmetric
space. Note that this is what one can expect since $S^2$ only
admits the zero homogeneous structure tensor \cite{TV}.

\bigskip

For the case $\lambda=0$ one can follow the proof of
Ambrose-Singer's Theorem to construct the Lie algebra of a group
acting transitively on $S^3$. As a computation shows the holonomy
of the connection $\wbarnabla=\bar{\nabla}-\bar{S}_0$ is trivial,
and one obtains the reductive decomposition
$T_eS^3\oplus\{0\}\simeq\f{su}(2)$ which describes the action of
$SU(2)\simeq S^3$ on itself. We then have  an example of a
homogeneous Riemannian structure $\bar{S}^0$ satisfying
$\wbarnabla \omega=\alpha\cdot\omega$ as in Theorem \ref{Th2}
($\omega$ being the mechanical connection form of the Hopf
fibration $S^3\to S^2$), but for which the structure group of the
fibration ($H=U(1)$) can not be seen as a normal subgroup of the
group ($\bar{G}'=SU(2)$) obtained by the proof of Ambrose-Singer's
Theorem .

\begin{remark}
\emph{There are not more reducible tensors than those described
above as the other groups acting transitively on $S^3$ are
$SO(4)$, which has no normal subgroups, and $SU(2)\simeq S^3$. In
addition, this procedure can be adapted to the Berger $3$-spheres,
where a family of homogeneous structures is calculated in
(\cite{GO2}). All reducible structures of this family reduce to
$S=0$ on $S^2$ as expected.}
\end{remark}

\begin{remark}
\emph{The groups acting isometrically and transitively on $S^7$
(see \cite{Salamon}) are $SO(7)$, $SU(4)$, $Sp(2)Sp(1)$, $U(4)$
and $Sp(2)U(1)$. The first two groups do not have normal subgroups
and hence do not fit in the reduction scheme. The group
$\bar{G}=Sp(2)Sp(1)$ has the normal subgroup $H=Sp(1)=SU(2)$,
which gives the Hopf fibration $S^7 \to S^4$. In this case, a
similar computation to the fibration $S^3 \to S^2$ shows that the
corresponding homogeneous Riemannian structures in the 7-sphere
reduce to the null tensor on $S^4$, the only homogeneous structure
in the four dimensional sphere. The last two groups are analized
in the following subsection.}
\end{remark}

\subsubsection{The fibration $S^7\to \mathbb{C}P^3$}

Let $\Delta^i_j$ denote the $4\times 4$ complex matrix with $1$ in
the $i$-th row and the $j$-th column and the rest zeros. Let $S^7$
be the standard $7$-sphere as a Riemannian sub-manifold of
$\mathbb{C}^4$ with the usual Hermitian inner product. The
standard action of the unitary group $U(4)$ on $\mathbb{C}^4$
gives a transitive and effective action on $S^7$ by isometries.
The isotropy group $\bar{K}$ at $\bar{x}=(1,0,0,0)\in\mathbb{S}^7$
is isomorphic to $U(3)$ and we can decompose
$\f{u}(4)=\bar{\f{m}}\oplus\bar{\f{k}}$ where
$$\bar{\f{k}}=\left\{\begin{pmatrix} 0 & 0 \\ 0 & A \end{pmatrix}/A\in\f{u}(3)\right\}$$
and
$$\bar{\f{m}}=\mathrm{span}\{i\Delta^1_1,\Delta^1_j-\Delta^j_1,i(\Delta^1_j+\Delta^j_1),j=1,2,3\}.$$
One can check that $\f{u}(4)=\bar{\f{m}}\oplus\bar{\f{k}}$ is the
unique reductive decomposition of $\f{u}(4)$ with respect to
$\bar{\f{k}}$. From (\ref{formula tensor (3,0)}), identifying
$\mathbb{R}^8\simeq \mathbb{C}^4$ and taking its natural
coordinates $(\bar{x}^1,\ldots,\bar{x}^8)$, the expression of the
homogeneous structure tensor $\bar{S}$ associated to this
decomposition at $T_{\bar{x}}S^7$ reads
\begin{eqnarray}\label{tensor S7 u(4)}
\bar{S}_{\bar{x}} & = & d\bar{x}^3\otimes
d\bar{x}^2\wedge d\bar{x}^4-d\bar{x}^4\otimes d\bar{x}^2\wedge
d\bar{x}^3+d\bar{x}^5\otimes
d\bar{x}^2\wedge d\bar{x}^6\nonumber\\
 & & -d\bar{x}^6\otimes
d\bar{x}^2\wedge d\bar{x}^5 +d\bar{x}^7\otimes d\bar{x}^2\wedge
d\bar{x}^8-d\bar{x}^8\otimes d\bar{x}^2\wedge d\bar{x}^7.
\end{eqnarray}
As a simple computation shows, this tensor belongs to the class
$\mathcal{S}_2\oplus\mathcal{S}_3$.

Let $H$ be the subgroup of $U(4)$ isomorphic to $U(1)$ given by
$$H=\left\{z\cdot I /z\in U(1)\right\}$$ where
$I$ is the $4\times 4$ identity matrix. It is obvious that $H$
is a normal subgroup of $U(4)$ the action of which on $S^7$ is
free. The reduction of $S^7$ by the action of $H$ gives the Hopf
fibration $S^7\rarrow\mathbb{C}P^3$ with which the
complex projective space inherits the Fubiny-Study metric. The
vertical and horizontal subspaces at $\bar{x}$ are
$$V_{\bar{x}}S^7=\mathrm{span}\left\{\frac{\partial}{\partial\bar{x}_2}\right\},
\hspace{1em}H_{\bar{x}}S^7=\mathrm{span}\left\{\frac{\partial}{\partial\bar{x}_3},
\ldots,\frac{\partial}{\partial\bar{x}_8}\right\}.$$ As in the
Hopf fibration $S^3 \to S^2$, the homogeneous structure tensor
$\bar{S}$ reduces to $S=0$, describing
$$\mathbb{C}P^3=\frac{U(4)}{U(3)\times U(1)}$$ as a symmetric
space.

\bigskip

If $\mathbb{H}$ denotes the quaternion algebra, we now see the
$7$-sphere
$$S^7=\left\{\begin{pmatrix}q_1\\q_2\end{pmatrix}\in\mathbb{H}^2/|q_1|^2+|q_2|^2=1\right\}$$ as a Riemannian
sub-manifold of $\mathbb{H}^2$ with the standard quaternion inner
product. The group $Sp(2)U(1)=Sp(2)\times_{\mathbb{Z}_2}U(1)$ acts
on $\mathbb{H}^2$ by
$$(A,z)\cdot\begin{pmatrix}q_1\\q_2\end{pmatrix}=A\begin{pmatrix}q_1\overline{z}\\q_2\overline{z}\end{pmatrix},\hspace{1em}\begin{pmatrix}q_1\\q_2\end{pmatrix}\in\mathbb{H}^2,A\in Sp(2),z\in U(1)$$
where $\overline{z}$ stands for the complex conjugation. This
action restricts to a transitive and effective action by
isometries on $S^7$. The isotropy group at $\bar{x}=(1,0)\in S^7$
is
$$\bar{K}=\left\{\left(\begin{pmatrix}z&0\\0&q\end{pmatrix},z\right)/q\in Sp(1),z\in U(1)\right\}/\mathbb{Z}_2$$
which is isomorphic to $Sp(1)U(1)$. Let $i,j,k$ be the imaginary
quaternion units and $i$ be the imaginary complex unit. Then, the
Lie algebra of $Sp(2)U(1)$ is $\f{sp}(2)\oplus\f{u}(1)$ where
\begin{eqnarray*}
\f{sp}(2)&=&\mathrm{span}\left\{\begin{pmatrix}0&1\\-1&0\end{pmatrix},\begin{pmatrix}ç
i&0\\0&0\end{pmatrix},\begin{pmatrix}0&i\\i&0\end{pmatrix},\begin{pmatrix}j&0\\0&0\end{pmatrix},\begin{pmatrix}0&j\\j&0\end{pmatrix}\right.\\
& &
\hspace{1.2cm}\left.\begin{pmatrix}k&0\\0&0\end{pmatrix},\begin{pmatrix}0&k\\k&0\end{pmatrix},\begin{pmatrix}0&0\\0&i\end{pmatrix},\begin{pmatrix}0&0\\0&j\end{pmatrix},\begin{pmatrix}0&0\\0&k\end{pmatrix}\right\}
\end{eqnarray*}
and $\f{u}(1)=\mathrm{span}\{i\}$; and then the isotropy algebra
is
$$\bar{\f{k}}=\mathrm{span}\left\{\begin{pmatrix}i&0\\0&0\end{pmatrix}+i,\begin{pmatrix}0&0\\0&i\end{pmatrix},\begin{pmatrix}0&0\\0&j\end{pmatrix},\begin{pmatrix}0&0\\0&k\end{pmatrix}\right\}.$$
Taking
\begin{eqnarray*}
\bar{\f{m}}&=&\mathrm{span}\left\{\begin{pmatrix}0&1\\-1&0\end{pmatrix},\begin{pmatrix}i&0\\0&0\end{pmatrix},\begin{pmatrix}0&i\\i&0\end{pmatrix}\right.\\
& &
\left.\begin{pmatrix}j&0\\0&0\end{pmatrix},\begin{pmatrix}0&j\\j&0\end{pmatrix},\begin{pmatrix}k&0\\0&0\end{pmatrix},\begin{pmatrix}0&k\\k&0\end{pmatrix}\right\}
\end{eqnarray*}
we have that
$\f{sp}(2)\oplus\f{u}(1)=\bar{\f{m}}\oplus\bar{\f{k}}$
 is a reductive decomposition. All other reductive decompositions
associated to $\f{sp}(2)\oplus\f{u}(1)$ and $\bar{\f{k}}$ are
given by a one-parameter family of complements
$\bar{\f{m}}_{\lambda}$, $\lambda\in\mathbb{R}$, which are the
graph of the $Ad(\bar{K})$-equivariant maps
$\varphi_{\lambda}:\bar{\f{m}}\rarrow\bar{\f{k}}$, where
$\varphi_{\lambda}$ maps $\begin{pmatrix}i&0\\0&0\end{pmatrix}$ to
$\lambda\begin{pmatrix}i&0\\0&0\end{pmatrix}+\lambda i$ and the
rest of elements of the basis to zero. Identifying
$\mathbb{H}^2\equiv\mathbb{R}^8$, the homogeneous structure tensor
$\bar{S}^{\lambda}$ associated to each reductive decomposition
$\f{sp}(2)\oplus\f{u}(1)=\bar{\f{m}}_{\lambda}\oplus\bar{\f{k}}$
is computed at $T_{\bar{x}}S^7$ as
\begin{eqnarray*}
(\bar{S}^{\lambda})_{\bar{x}} & = &
d\bar{x}^5\otimes d\bar{x}^2\wedge d\bar{x}^6+d\bar{x}^5\otimes d\bar{x}^3\wedge d\bar{x}^7+d\bar{x}^5\otimes d\bar{x}^4\wedge d\bar{x}^8\\
 & &
 -\lambda d\bar{x}^2\otimes d\bar{x}^5\wedge d\bar{x}^6+(1+2\lambda)d\bar{x}^2\otimes d\bar{x}^3\wedge d\bar{x}^4+\lambda d\bar{x}^2\otimes d\bar{x}^7\wedge d\bar{x}^8\\
  & &
  +d\bar{x}^6\otimes d\bar{x}^5\wedge d\bar{x}^2+d\bar{x}^6\otimes d\bar{x}^3\wedge d\bar{x}^8-d\bar{x}^6\otimes d\bar{x}^4\wedge d\bar{x}^7\\
   & &
   +d\bar{x}^3\otimes d\bar{x}^2\wedge d\bar{x}^4+d\bar{x}^4\otimes d\bar{x}^2\wedge d\bar{x}^3\\
   & &
   -d\bar{x}^7\otimes d\bar{x}^3\wedge d\bar{x}^5-d\bar{x}^7\otimes d\bar{x}^2\wedge d\bar{x}^8+d\bar{x}^7\otimes d\bar{x}^4\wedge d\bar{x}^6\\
   & &
   -d\bar{x}^8\otimes d\bar{x}^4\wedge d\bar{x}^5+d\bar{x}^8\otimes d\bar{x}^2\wedge d\bar{x}^7-d\bar{x}^8\otimes d\bar{x}^3\wedge d\bar{x}^6.
\end{eqnarray*}

Let $H=\{(Id,w)/w\in U(1)\}\subset Sp(2)U(1)$, where $Id$ is the
identity of $Sp(2)$, it is easy to see that $H$ is a normal
subgroup of $Sp(2)U(1)$ isomorphic to $U(1)$. Reduction by
the action of $H$ gives again the Hopf fibration $\pi:S^7\rarrow \mathbb{C}P^3$ with
$\pi(\bar{x})=[1:0:0:0]\in\mathbb{C}P^3$. The vertical and
horizontal subspaces of $\pi$ at $\bar{x}$ are
$$V_{\bar{x}}S^7=\mathrm{span}\left\{\frac{\partial}{\partial\bar{x}_2}\right\},\hspace{1em}
H_{\bar{x}}S^7=\mathrm{span}\left\{\frac{\partial}{\partial\bar{x}_3},\ldots,\frac{\partial}{\partial\bar{x}_8}\right\}$$
Let $(t^1,\ldots,t^6): \mathbb{C}P^3-\{z_0=0\} \rarrow
\mathbb{R}^6$ be the coordinate system around $x=[1:0:0:0]$ given
by
\[
[z_0\colon z_1 \colon z_2 \colon z_3] \mapsto
\left(\mathrm{Re}\left(\tfrac{z_1}{z_0}\right),
\mathrm{Im}\left(\tfrac{z_1}{z_0}\right),\mathrm{Re}\left(\tfrac{z_2}{z_0}\right),
\mathrm{Im}\left(\tfrac{z_2}{z_0}\right)
,\mathrm{Re}\left(\tfrac{z_3}{z_0}\right),\mathrm{Im}\left(\tfrac{z_3}{z_0}\right)\right).
\]
The reduced homogeneous structure tensor $S$ is computed at
$T_x\mathbb{C}P^3$ as
\begin{eqnarray*}
S_x & = & dt^3\otimes dt^1\wedge dt^5+dt^3\otimes dt^2\wedge dt^6\\
    & + & dt^4\otimes dt^1\wedge dt^6-dt^4\otimes dt^2\wedge dt^5\\
    & + & dt^5\otimes dt^2\wedge dt^4-dt^5\otimes dt^1\wedge dt^3\\
    & - & dt^6\otimes dt^2\wedge dt^3-dt^6\otimes dt^1\wedge dt^4.
\end{eqnarray*}
It is easy to check that $\bar{S}^{\lambda}$ is a
$\mathcal{S}_2\oplus\mathcal{S}_3$ structure for all
$\lambda\in\mathbb{R}$ which is not $\mathcal{S}_2$ nor
$\mathcal{S}_3$ for any $\lambda$, and $S$ is also a strict
$\mathcal{S}_2\oplus\mathcal{S}_3$ structure. Note that in the
latter and the previous example the class
$\mathcal{S}_2\oplus\mathcal{S}_3$ is preserved by the reduction
procedure. This fact is expected from Proposition \ref{proposicion
minimal} since the fibres of the Hopf fibration are totally
geodesic and in particular minimal Riemannian sub-manifolds of
$S^7$.

\section{Almost contact metric-almost Hermitian and Sasakiann-K\"ahler reduction}

An \textit{almost contact structure} on a manifold $\bar{M}$ is a
triple $(\phi,\xi,\eta)$ where $\phi$ is a $(1,1)$-tensor field,
$\xi$ is a vector field, and $\eta$ is a $1$-form satisfying
$$\begin{array}{c}
\phi(\xi)=0,\hspace{1em}\eta(\phi(\bar{X}))=0,\hspace{1em}\eta(\xi)=1,\\
\phi^2=-\mathrm{id}+\eta\otimes\xi ,
\end{array}$$
for all $\bar{X}\in\f{X}(\bar{M})$. The almost contact structure
is said to be \textit{strictly regular} if $\xi$ is a regular
vector field such that all orbits of which are homeomorphic, and
\textit{invariant} if $\phi$ and $\eta$ are invariant by the
action of the one parameter group of $\xi$. In the following all
almost contact structures are supposed to be invariant and
strictly regular. In \cite{Ogiue} the following results are
proved:

\begin{theorem}\label{thm reduccion contact Oigue}
Let $(\phi,\xi,\eta)$ be an almost contact structure and $M$ the
space of orbits given by $\xi$. Then $M$ is endowed with a smooth
structure such that $\pi:\bar{M}\rarrow M$ is an principal bundle
and $\eta$ is a connection form.
\end{theorem}

\begin{theorem}
In the situation of the previous Theorem, the $(1,1)$-tensor field $J$ defined in $M$ by
$$J_xX=\pi_*(\phi_{\bar{x}}X^H),\hspace{1em}x\in M,X\in\f{X}(M),$$
where $\bar{x}\in \pi ^{-1}(x)\subset \bar{M}$ and $X^H$ is the
horizontal lift of $X$ with respect to $\eta$, is an almost
complex structure.
\end{theorem}

If $\bar{M}$ is equipped with a Riemannian metric $\bar{g}$, an
almost contact structure $(\phi,\xi,\eta)$ is said to be {\it
metric} if the following conditions hold
$$\bar{g}(\xi,\bar{X})=\eta(\bar{X}),\hspace{1em}\bar{g}
(\phi\bar{X},\phi\bar{Y})=\bar{g}(\bar{X},\bar{Y})+\eta(\bar{X})\eta(\bar{Y}).$$
Note that this implies that $\eta$ defines the mechanical
connection in $(\bar{M},\bar{g})\rarrow M$ and induces a
Riemannian metric $g$ in $M$. In this situation it can be proved
\cite{Ogiue} that $(J,g)$ is almost Hermitian. Let
$\Phi(\bar{X},\bar{Y})=\bar{g}(\phi\bar{X},\bar{Y})$ be the
fundamental $2$-form of the almost contact metric structure, then
$(\phi,\xi,\eta,g)$ is called an \textit{almost Sasakian}
structure if $d\eta=2\Phi$. If moreover
$\bar{\nabla}\phi=\bar{g}\otimes\xi-\mathrm{id}\otimes\eta $ where
$\bar{\nabla}$ is the Levi-Civita connection of $\bar{g}$, then it
is called a \textit{Sasakian} structure. It can be proved
\cite{Ogiue} that if $(\phi,\xi,\eta,g)$ is (almost) Sasakian then
$(J,g)$ is (almost) K\"ahler.

An almost contact metric manifold is called homogeneous almost
contact metric if there is a transitive group of isometries such
that $\phi$ is invariant (and then also $\xi$ and $\eta$). If the
manifold is (almost) Sasakian then it is called (almost) Sasakian
homogeneous. A homogeneous structure tensor $\bar{S}$ on $\bar{M}$ is
called a homogeneous almost contact metric structure if
$\wbarnabla\phi=0$ (and then $\wbarnabla\xi=0$ and
$\wbarnabla\eta=0$). From the result of Kiri\v{c}enko \cite{Kir}
we have that a connected, simply connected and complete Riemannian
manifold is a homogeneous almost contact metric manifold if and
only if it admits a homogeneous almost contact metric structure.
If the manifold is (almost) Sasakian then it is homogeneous
(almost) Sasakian if and only if it admits a homogeneous (almost)
Sasakian structure.

We now assume that $\bar{S}$ is an almost contact metric
homogeneous structure invariant by the one parameter group of
$\xi$. Since $\wbarnabla\eta=0$, we are in the situation of
Theorem \ref{Th2} and then the tensor
$S_XY=\pi_*(\bar{S}_{X^H}Y^H)$ defines a homogenous structure on
$M$.

\begin{proposition}\label{ultprop}

The reduced homogeneous structure $S$ in $M$ is a homogeneous
almost Hermitian structure on $M$. Moreover, if $\bar{S}$ is
homogeneous (almost) Sasakian structure, then the reduced
homogeneous structure $S$ is a homogeneous (almost) K\"ahler
structure on $M$.
\end{proposition}

\begin{proof}
Let $\wnabla=\nabla-S$, where $\nabla$ is the Levi-Civita
connection of $g$. Then $\wnabla_XY=\pi_*(\wbarnabla_{X^H}Y^H)$.
Since $\eta(\phi(\bar{X}))=0$ we have that $\phi(\bar{X})$ is
horizontal for all $\bar{X}\in\f{X}(\bar{M})$. For any
$X,Y\in\f{X}(M)$ we have
\begin{eqnarray*}
\left(\wnabla_XJ\right)Y & = & \wnabla_X(JY)-J\left(\wnabla_XY\right)\\
 & = & \pi_*\left(\wbarnabla_{X^H}(JY)^H\right)-\pi_*\left(\phi\left(\wbarnabla_{X^H}Y^H\right)\right)\\
 & = & \pi_*\left(\wbarnabla_{X^H}\left(\phi Y^H\right)-\phi\left(\wbarnabla_{X^H}Y^H\right)\right)\\
 & = & \pi_*\left(\left(\wbarnabla_{X^H}\phi\right)Y^H\right)\\
 & = & 0
\end{eqnarray*}
and hence $\wnabla J=0$.
\end{proof}

\bigskip

We now apply Proposition \ref{ultprop} to the Hopf fibrations $S^3
\to S^2$ and $S^7 \to \mathbb{C}P^3$ and check that the
Sasakian-K\"ahler reduction procedure gives the null K\"ahler
structures of the reduced spaces, the only homogeneous K\"ahler
structures existing on $S^2$ and $\mathbb{C}P^3$. For the first
case, let $(\bar{x}^1,\bar{x}^2,\bar{x}^3,\bar{x}^4)$ be the
natural coordinates of $\mathbb{R}^4$ and
$$\alpha=-\bar{x}^2d\bar{x}^1+\bar{x}^1d\bar{x}^2-\bar{x}^4d\bar{x}^3+\bar{x}^3d\bar{x}^4.$$
If $i:S^3\rarrow \mathbb{R}^4$ is the natural immersion of the
Euclidean 3-sphere in $\mathbb{R}^4$, the form $\eta=i^*\alpha$
defines an almost contact metric structure on $S^3$ which is
moreover a Sasakian structure \cite{Blair}. One can check (see
\cite{GO2}) that the homogeneous Sasakian structures on $S^3$ with
respect to $\eta$ are those given in (\ref{tensor S3 u(2)}) after
the isometry
$$\begin{array}{rrcl}
\varphi: & S^3 & \longrightarrow & S^3\\
 & (\bar{x}^1,\bar{x}^2,\bar{x}^3,\bar{x}^4) & \mapsto &
  (\bar{x}^1,-\bar{x}^2,-\bar{x}^3,-\bar{x}^4),
\end{array}$$
namely
\begin{equation}\label{family sasakian homogeneous struct S^3}
(\bar{S}^{\lambda})_{\bar{x}}=(1-\lambda)d\bar{x}^2\otimes
d\bar{x}^3\wedge d\bar{x}^4-d\bar{x}^3\otimes d\bar{x}^2\wedge
d\bar{x}^4+d\bar{x}^4\otimes d\bar{x}^2\wedge
d\bar{x}^3.
\end{equation}
This homogeneous structures are obtained from the group of isometries
$$G=\{\varphi\circ \Phi_{a}\circ \varphi^{-1}/a\in U(2)\}$$
where $\Phi_{a}$ denotes the standard action of $U(2)$ on $S^3$. The subgroup
$$H=\{\varphi\circ \Phi_{z}\circ \varphi^{-1}/z\in U(1)\}$$ is a normal
subgroup of $G$, where $z\in U(1)$ is seen in $U(2)$ as the
matrix
$\begin{pmatrix}z& 0\\ 0&z\end{pmatrix}$. Reduction by the
action of $H$ gives the fibration
$$\begin{array}{rcl}
S^3 & \rarrow & S^2\\
(z_1,z_2) & \mapsto & (2z_1z_2,|z_1|^2-|z_2|^2)
\end{array}$$
which is precisely the fibration given by the Sasakian structure $\eta$ in the sense of Theorem \ref{thm reduccion contact Oigue}.
The reduction described in Proposition
\ref{ultprop} by the action of $H$ of the family of homogeneous structures (\ref{family sasakian homogeneous struct S^3}) is (as we had in \S
\ref{sectHopf}) the tensor $S=0$.

\bigskip

As for the second fibration, we take
$(\bar{x}^1,\ldots,\bar{x}^8)$ the coordinates of $\mathbb{R}^8$
and
$$\alpha=-\bar{x}^2d\bar{x}^1+\bar{x}^1d\bar{x}^2-\bar{x}^4d\bar{x}^3+\bar{x}^3d\bar{x}^4-
\bar{x}^6d\bar{x}^5+\bar{x}^5d\bar{x}^6-\bar{x}^8d\bar{x}^7+\bar{x}^7d\bar{x}^8.$$
The form $\eta=i^*\alpha$, where $i:S^7\rarrow \mathbb{R}^8$ is
the natural immersion of the Euclidean 7-sphere, defines an almost
contact metric structure on $S^7$ which is moreover Sasakian (cf.
\cite{Blair}). A homogeneous Sasakian structure on $S^7$ with
respect to $\eta$ is obtained by transforming (\ref{tensor S7
u(4)}) with respect to the isometry
$$\begin{array}{rrcl}
\varphi: & S^7 & \longrightarrow & S^7\\
 & (\bar{x}^1,\ldots,\bar{x}^8) & \mapsto & (\bar{x}^1,-\bar{x}^2,\ldots,-\bar{x}^8),
\end{array}$$
and reads
\begin{eqnarray}\label{family sasakian homogeneous struct S^7}
\bar{S}_{\bar{x}} & = & -d\bar{x}^3\otimes
d\bar{x}^2\wedge d\bar{x}^4+d\bar{x}^4\otimes d\bar{x}^2\wedge
d\bar{x}^3-d\bar{x}^5\otimes
d\bar{x}^2\wedge d\bar{x}^6\nonumber\\
 & & +d\bar{x}^6\otimes
d\bar{x}^2\wedge d\bar{x}^5 -d\bar{x}^7\otimes d\bar{x}^2\wedge
d\bar{x}^8+d\bar{x}^8\otimes d\bar{x}^2\wedge d\bar{x}^7.
\end{eqnarray}
This family of homogeneous structure tensors are also obtained from the action of the group of isometries
$$G=\{\varphi\circ \Phi_{a}\circ \varphi^{-1}/a\in U(4)\}$$
where $\Phi_{a}$ denotes the standard action of $U(4)$ on $S^7$. The subgroup
$$H=\{\varphi\circ \Phi_{z}\circ \varphi^{-1}/z\in U(1)\}$$ is a
normal subgroup of $G$, and reduction by the action of $H$
provides
the fibration given by the Sasakian structure $\eta$ in
the sense of Theorem \ref{thm reduccion contact Oigue}. Again, the
family (\ref{family sasakian homogeneous struct S^7}) reduces to
$S=0$.

\bigskip

A non trivial projection of homogeneous Sasakian structure tensors
can be found in the following situation. Let $\pi : \bar{M}\to
\mathbb{C}H(n)$ be a principal line bundle endowed with the
Sasakian structure $(\phi,\xi,\eta,\bar{g})$ given by an invariant
metric $\bar{g}$ and its corresponding mechanical connection
$\eta$ in $\bar{M}$ (see \cite{GO}). Then, every homogeneous
K\"ahler structure tensor $S$ in $\mathbb{C}H(n)$ can be obtained
as the reduction of the Sasakian homogeneous structure tensor
\[
\bar{S}_{X^H}Y^H=(S_XY)^H-\bar{g}(X^H,\phi Y^H)\xi,\quad
\bar{S}_{X^H}\xi=-\phi X^H=\bar{S}_\xi X^H,\quad \bar{S}_\xi
\xi=0,
\]
in the sense of Proposition \ref{ultprop}. The description of all
these tensors have been previously studied in \cite{GO}.
Nevertheless, it is interesting to point out that the goal of that
reference was the lift of structures from $\mathbb{C}H(n)$ to
$\bar{M}$. The result given in Proposition \ref{ultprop} thus
gives a reverse procedure of that particular situation.

\section*{Acknowledgements}
The authors are deeply indebted to Prof. Andrew Swann and Prof.
P.M. Gadea for useful conversations about the topics of this
paper.

This work has been partially funded by Ministerio de Ciencia e
Innovaci\'on (Spain) under project MTM2011-22528.

\end{document}